\documentclass[11pt]{amsart}
\usepackage{graphicx}
\usepackage{amssymb}
\usepackage{amsmath}
\usepackage{amsthm,amsfonts,bbm}
\usepackage{amscd}
\usepackage{geometry}
\usepackage[all,2cell]{xy}
\usepackage{epsfig,epstopdf}
\usepackage{mathrsfs}
\usepackage{hyperref}
\usepackage{cite}
\usepackage{color}

\UseAllTwocells \SilentMatrices

\newtheorem{thm}{Theorem}[section]
\newtheorem{cor}[thm]{Corollary}
\newtheorem{lem}[thm]{Lemma}
\theoremstyle{definition}
\newtheorem{defi}[thm]{Definition}
\theoremstyle{remark}

\newtheorem{prob}[thm]{\bf Problem}
\numberwithin{equation}{section}
\numberwithin{figure}{section}
\def\A{\mathcal{A}}

\def \CC{\mathscr{C}}
\def \F{\mathcal{F}}
\def \G{\mathcal{G}}
\def \H{\mathcal{H}}
\def\I{\mathcal{I}}
\def\la{\lambda}
\def \omg{\omega}
\def \R{\mathscr{R}}

\def \T{\mathcal{T}}
\def \Tr{\text{Tr}}
\def \tr{\text{tr}}
\def \U{\mathcal{U}}
\def \V{\mathscr{V}}
\def \W{\mathbf{W}}
\def \WW{\mathscr{W}}
\def\Z{\mathbb{Z}}

\begin{document}
\title[Trace of hypergraph]{The trace of uniform hypergraphs with application to Estrada index}

\author[Y.-Z. Fan]{Yi-Zheng Fan*}
\address{Center for Pure Mathematics, School of Mathematical Sciences, Anhui University, Hefei 230601, P. R. China}
\email{fanyz@ahu.edu.cn}
\thanks{*The corresponding author.
This work was supported by National Natural Science Foundation of China (Grant No. 11871073).}

\author[J. Zheng]{Jian Zheng}
\address{School of Mathematical Sciences, Anhui University, Hefei 230601, P. R. China}
\email{zhengj@stu.ahu.edu.cn}

\author[Y. Yang]{Ya Yang}
\address{School of Mathematical Sciences, Anhui University, Hefei 230601, P. R. China}
\email{yangy@stu.ahu.edu.cn}

\subjclass[2000]{Primary 05C65, 15A69; Secondary 13P15, 14M99}

\keywords{Hypergraph; trace; Estrada index; adjacency tensor; eigenvalue}

\begin{abstract}
In this paper we investigate the traces of the adjacency tensor of hypergraphs (simply called the traces of hypergraphs).
We give new expressions for the traces of hypertrees and linear unicyclic hypergraphs by the weight function assigned to their connected sub-hypergraphs, and provide some perturbation results for the traces of a hypergraph with cut vertices.
As applications we determine the unique hypertree with maximum Estrada index among all hypertrees with fixed number of edges and perfect matchings, and the unique unicyclic hypergraph with maximum Estrada index among all unicyclic hypergraph with fixed number of edges and girth $3$.
\end{abstract}

\maketitle

\section{Introduction}
A  \emph{hypergraph} $\H=(V,E)$ consists of a vertex set $V=\{v_1,v_2,{\cdots},v_n\}$ denoted by $V(\H)$ and an edge set $E=\{e_1,e_2,{\cdots},e_k\}$ denoted by $E(\H)$,
 where $e_i \subseteq V$ for $i \in [k]$.
 If there exist no different $i$ and $j$ such that $e_i \subseteq e_j$, then $\H$ is called \emph{simple}.
 If $|e_i|=m$ for each $i \in [k]$ and $m \geq2$, then $\H$ is called an \emph{$m$-uniform} hypergraph.
 A simple graph is exactly a simple $2$-uniform hypergraph.
 %Throughout this paper, all hypergraphs are considered simple and $m$-uniform.

 For an $m$-uniform hypergraph $\H$ on vertices $v_1,\ldots,v_n$.
 Cooper and Dutle \cite{CD2012} introduced the adjacency tensor of $\H$ as follows.
\begin{defi}(\cite{CD2012})
Let $\mathcal{H}$ be an $m$-uniform hypergraph on $n$ vertices $v_{1},v_{2},\ldots,v_{n}$.
The {\it adjacency tensor} of $\mathcal{H}$ is defined as $\mathcal{A}(\mathcal{H})=(a_{i_{1}i_{2}\ldots i_{m}})$, an $m$-th order $n$-dimensional tensor, where
$$a_{i_{1} i_{2} \ldots i_{m}}=\left\{
\begin{array}{cl}
\frac{1}{(m-1)!}, & \mbox{~if~} \{v_{i_{1}},\ldots,v_{i_{m}}\} \in E(H);\\
0, & \mbox{~else}.
\end{array}\right.
$$
\end{defi}
 Note that if $m=2$, then $\A(\H)$ is exactly the usual adjacency matrix of the simple graph $\H$.
 In this situation, the $d$-th trace of $\A(\H)$, namely the trace of $\A(\H)^d$, is exactly the number of closed walks of $\H$ with length $d$ starting from each vertex of $\H$.

 To deal with the high order case, Morozov and Shakirov \cite{MS2011} introduced the traces of polynomial maps $f$ given by a system of homogeneous polynomials of arbitrary degrees.
 As a tensor $\mathcal{T}=(t_{i_{1}i_{2}\ldots i_{m}})$ of order $m$ and dimension $n$ naturally induces a system of homogeneous polynomials $\T x^{m-1}$ defined by
 $$ (\T x^{m-1})_i=\sum_{i_{2},\ldots,i_{m}\in [n]}t_{ii_{2}\ldots i_{m}}x_{i_{2}}\cdots x_{i_m}, i \in [n],$$
 where $x=(x_1,\ldots,x_n)^\top$.
 Using the traces defined by Morozov and Shakirov \cite{MS2011}, the $d$-th trace $\Tr_d(\T)$ of $\T$  is expressed as follow:
\begin{equation}\label{MSeq}\Tr_d(\T)=(m-1)^{n-1} \sum_{d_1+\cdots+d_n=d, \atop d_i \in \mathbb{N}, i \in [n]}
\prod_{i=1}^n \frac{1}{(d_i(m-1))!} \left(\sum_{y_i \in [n]^{m-1}}t_{iy_i}\frac{\partial }{\partial a_{iy_i}}\right)^{d_i}
\Tr(A^{d(m-1)}),
\end{equation}
where $A=(a_{ij})$ be an $n \times n$ auxiliary matrix by taking all $a_{ij}$'s as variables, $t_{iy_i}=t_{ii_2 \ldots i_m}$ and $\frac{\partial }{\partial a_{iy_i}}=\frac{\partial }{\partial a_{ii_2}} \cdots \frac{\partial }{\partial a_{ii_m}}$ if
$y_i=(i_2, \ldots, i_m)$.

The traces of a tensor are closely related to its eigenvalues, which were introduced by Lim \cite{Lim} and Qi \cite{Qi} as follows, where $\mathcal{I}=(i_{i_{1}i_{2}\ldots i_{m}})$ is the {\it identity tensor} of order $m$ and dimension $n$, namely, $i_{i_{1}i_{2} \ldots i_{m}}=1$ if
   $i_{1}=i_{2}=\cdots=i_{m} \in [n]$ and $i_{i_{1}i_{2} \ldots i_{m}}=0$ otherwise.
\begin{defi}(\cite{Lim, Qi}\label{21}) Let $\mathcal{T}$ be an $m$-th order $n$-dimensional tensor.
For some $\lambda \in \mathbb{C}$, if the polynomial system $(\lambda \mathcal{I}-\mathcal{T})x^{m-1}=0$, or equivalently $\mathcal{T}x^{m-1}=\lambda x^{[m-1]}$, has a solution $x\in \mathbb{C}^{n}\backslash \{0\}$,
then $\lambda $ is called an {\it eigenvalue} of $\mathcal{T}$ and $x$ is an {\it eigenvector} of $\mathcal{T}$ associated with $\lambda$,
where $x^{[m-1]}:=(x_{1}^{m-1}, x_{2}^{m-1},\ldots,x_{n}^{m-1})$.
\end{defi}

The \emph{determinant} $\det \T$ of $\T$ is defined to be the resultant of the polynomials $\T x^{m-1}$ \cite{Ha},
and the \emph{characteristic polynomial} of $\T$ is defined to be $\varphi_\T(\la):=\det(\la \I-\T)$ \cite{Qi,CPZ2}.
It is known that $\la$ is an eigenvalue of $\T$ if and only if it is a root of $\varphi_\T(\la)$.
The \emph{spectrum} of $\T$ is the multi-set of the roots of $\varphi_\T(\la)$.

Morozov and Shakirov proved that
\begin{equation}\label{MS} \det(\I-\T)=\exp\left(\sum_{d=1}^\infty -\frac{\Tr_d(\T)}{d}\right)=\sum_{d=0}^\infty P_d\left(-\frac{\Tr_1(\T)}{1},\cdots, -\frac{\Tr_d(\T)}{d}\right),\end{equation}
where $P_d$ (the $d$th Schur polynomial) is defined as $P_0=1$, and for $d>0$,
$$P_d(t_1,\ldots,t_d)=\sum_{\ell=1}^d \sum_{d_1+\cdots+d_\ell=d,  d_i \in \mathbb{Z}^+, i \in [\ell]} \frac{t_{d_1}\cdots t_{d_\ell}}{\ell!}.$$
From Eq. (\ref{MS}), we can get the determinant $\det\T$ and the characteristic polynomial $\varphi_\T(\la)=\det(\la \I-\T)$ in terms of traces by considering the degree of the resultant; see \cite{CD2012,SQH2015} for details.

Cooper and Dulte \cite{CD2012} gave explicit formulas for some low co-degree coefficients of the characteristic polynomial of $\A(\H)$ of a uniform hypergraph $\H$.
Shao, Qi and Hu \cite{SQH2015} provided a graph interpretation for the $d$-th trace of a general tensor, and proved that
\begin{equation}\label{traeig} \Tr_d(\T)=\sum_{i=1}^N \la_i^d,\end{equation}
which is consistent with the matrix case,
where $\la_1,\ldots,\la_N$ are all eigenvalues of $\T$, and $N=n(m-1)^{n-1}$.

The \emph{$d$-th trace} of a uniform hypergraph $\H$, denoted by $\Tr_d(\H)$, is defined to be  $\Tr_d(\A(\H))$.
Clark and Cooper \cite{CC2021} generalized the Harary-Sachs theorem of graphs to uniform hypergraphs by expressing the trace as a weighted sum over a family of Veblen hypergraphs.
Chen, Bu and Zhou \cite{CBZ} gave a formula for the spectral moments (equivalently the traces) of a hypertree in terms of the number of sub-hypertrees.

The traces of a hypergraph are also related to the Estrada index of the hypergraph, which was recently introduced by Sun, Zhou and Bu \cite{SZB2021}.

\begin{defi}(\cite{SZB2021})\label{66}
Let $\mathcal{H}$ be an $m$-uniform hypergraph on $n$ vertices, and let $\lambda_{1},\ldots,\lambda_{N}$ be all eigenvalues
of the adjacency tensor $\mathcal{A}(\mathcal{H})$ of $\mathcal{H}$, where $N=n(m-1)^{n-1}$. The Estrada index of $\mathcal{H}$
is defined to be
$$EE(\mathcal{H})=\sum_{i=1}^{N}e^{\lambda_{i}}.$$
\end{defi}

By Eq. (\ref{traeig}), it is easily seen that
$$EE(\mathcal{H})=\sum_{d=0}^{\infty}\frac{Tr_{d}(\mathcal{H})}{d!}.$$
When $m=2$, the Estrada index in Definition \ref{66} is exactly that of a graph, which was first introduced by Estrada \cite{E2000} in 2000 and found useful in measuring the degree of protein folding \cite{E2002} and the centrality of complex networks \cite{ER2005}.
So the Estrada index of hypergraphs may has potential applications in networks modelled as hypergraphs.
Pe\~na, Gutman and Rada \cite{PGR2007} conjectured that
 the path is the unique graph with the minimum Estrada index among all graphs (trees) with given order, and the star is the unique one with the maximum Estrada index among all trees with given order.
 The conjecture was partly proved by Das and Lee \cite{DL2009}, and completely proved by Deng \cite{Deng2009}.
 The other versions of Estrada index of hypergraphs were also investigated
by Duan, Dam and  Wang \cite{DDW} via signless Laplacian tensor or Laplacian tensor, and Lu, Xue and Zhu \cite{LXZ} via signless Laplacian matrix.
Recently, Fan et al. \cite{FYa} proved that among all hypertrees with fixed number of edges, the hyperpath is the unique one with minimum Estrada index and the hyperstar is the unique one with maximum Estrada index, which provided a hypergraph version of the result of Pe\~na-Gutman-Rada conjecture.

We shall note here the development of spectral hypergraph theory.
Since the Perron-Frobenius theorem of nonnegative matrices was generalized to  nonnegative tensors \cite{CPZ1, FGH, YY1, YY2, YY3},
 the spectral hypergraph theory develops rapidly on many topics, such as the spectral radius \cite{BL,FanTPL, GCH2022,KLM2014,LSQ, LKS2018,LM}, the eigenvariety \cite{FBH,FTL,FHBproc}, the spectral symmetry \cite{FHB,FLW, SQH2015,Zhou}, the eigenvalues of hypertrees \cite{ZKSB}.

In this paper, we will give new expressions of the traces of uniform hypergraphs, especially for hypertrees and unicyclic hypergraphs, and provide some perturbation results for the traces of a hypergraph when its structure is locally changed.
As an application of the trace results, we determine the unique hypertree with maximum Estrada index among all hypertrees with fixed number of edges and perfect matching. We characterize the linear unicyclic hypergraphs with maximum Estrada index among all unicyclic hypergraphs with fixed number of edges and given girth, and particularly determine with maximum Estrada index among all unicyclic hypergraphs with fixed number of edges and girth $3$.

\section{Preliminaries}
\subsection{Tensors and hypergraphs}
A \emph{tensor} (also called \emph{hypermatrix}) $\T=(t_{i_{1} i_2 \ldots i_{m}})$ of order $m$ and dimension $n$ over $\mathbb{C}$ refers to a
 multiarray of entries $t_{i_{1}i_2\ldots i_{m}}\in \mathbb{C}$ for all $i_{j}\in [n]:=\{1,2,\ldots,n\}$ and $j\in [m]$, which can be viewed to be the coordinates of the classical tensor (as a multilinear function) under a certain basis.
 Surely, if $m=2$, then $\T$ is a matrix of size $n \times n$.

Let $\H$ be a hypergraph.
$\H$ is called \emph{trival} if it contains only one vertex; otherwise, it is called \emph{nontrivial}.
$\H$ is called \emph{linear} if any two different edges intersect into at most one vertex.
Let $v \in V(\H)$, and let $E_v(\H)$ denote the set of edges of $\H$ that contains the vertex $v$.
The \emph{degree}
$d_{v}(\mathcal{H})$ of $v$ in $\mathcal{H}$ is the cardinality of $E_v(\H)$.
A vertex $v$ of $\mathcal{H}$ is called a \emph{cored vertex} if it has degree one. An edge $e$ of $\mathcal{H}$ is called a \emph{pendent edge} if it contains $|e|-1$ cored vertices.
A \emph{walk} $W$ in $\mathcal{H}$ is a sequence of alternate vertices
and edges: $v_{0}e_{1}v_{1}e_{2}\cdots e_{l}v_{l}$, where $v_{i}\neq v_{i+1}$ and $\{v_{i},v_{i+1}\}\subseteq e_{i}$ for $i=0,1,\ldots,l-1$.
If $v_{0}=v_{l}$, then $W$ is called a {\it circuit}, and is called a {\it cycle} if no vertices or edges are repeated except $v_{0}=v_{l}$.
The hypergraph $\mathcal{H}$ is said to be \emph{connected} if every two vertices are connected by a walk; $\H$ is called a \emph{hypertree} if it is connected and acyclic, and is called \emph{unicyclic} if it contains exactly one cycle.

A \emph{matching} $M$ of $\mathcal{H}$ is a set of vertex-disjoint edges of $\mathcal{H}$, and $M$ is called a \emph{perfect matching} of $\mathcal{H}$ if it covers all vertices of $\H$.
A multi-hypergraph is a hypergraph allowed to have multiple edges, and is called
\emph{m-valent} if each vertex has degree of multiple of $m$.
A \emph{Veblen hypergraph} is an $m$-uniform and $m$-valent multi-hypergraph.
Throughout of this paper, all hypergraphs are consider simple unless stated somewhere.

Hu, Qi and Shao \cite{HQS} introduced a class of hypergraphs which are constructed from simple graph.

\begin{defi}(\cite{HQS})\label{pow}
Let $G=(V(G),E(G))$ be a graph with vertex set $V=V(G)$ and edge set $E=E(G)$. For an integer $m\geq3$, the
$m$-th power of $G$, denoted by $G^{m}:=(V^{m},E^{m})$, is defined to be the $m$-uniform hypergraph with vertex
set $V^{m}=V\cup\{i_{e,1},\ldots,i_{e,m-2}:e\in E\}$ and edge set $E^{m}=\{e\cup\{i_{e,1},\ldots,i_{e,m-2}\}:e\in E\}$, where
$i_{e,1},\ldots,i_{e,m-2}$ are new vertices inserted to each edge $e\in E(G)$.
\end{defi}

By Definition \ref{pow}, the power $T^m$ of a tree $T$ is a hypertree, and the power $U^m$ of a unicyclic graph $U$ is a linear unicyclic hypergraph.
Denote $P_n,C_n,S_n$ respectively a path, a cycle and a star all with $n$ edges (as simple graphs).
Then $C_n^m$ is a linear cycle as hypergraph, and $P_n^m$ is called a \emph{hyperpath}, and $S_n^m$ is called a \emph{hyperstar}.
The \emph{center} of $S_n$ or $S_n^m$ is the vertex with maximum degree.

\subsection{Traces of hypergraphs}
Shao, Qi and Hu \cite{SQH2015} gave a graph interpretation for the $d$-th trace $\Tr_d(\T)$.
Let
$$\F_d=\{((i_1,\alpha_1),\ldots,(i_d,\alpha_d)): i_1 \le \cdots \le i_d, \alpha_j \in [n]^{m-1}\},$$
 where $i_j$ is called the \emph{primary index} (or \emph{root}) of the $m$-tuple
$f_j:=(i_j,\alpha_j)$ for $j \in [d]$.
Define an $i_j$-rooted directed star for the tuple $f_j$:
$$ S_{f_j}(i_j)=(V_j, \{(i_j,u_k): k=1,\ldots,m-1 \}),$$
where $V_j=\{i_j,u_1,\ldots,u_{m-1}\}$ is considered as a set by omitting multiple indices if $\alpha_j=(u_1,\ldots,u_{m-1})$.
So we get a multi-directed graph associated with $F$, denoted and defined as
$$R(F)=\cup_{j=1}^d S_{f_j}(i_j).$$

Denote by $b(F)$ the product of the factorial of the multiplicities of all arcs of $R(F)$, $c(F)$ the product of the factorial of the outdegree of the all vertices of $R(F)$, and
$\W(F)$ the set of vertex sequences of all Euler tours of $R(F)$.
Shao, Qi and Hu \cite{SQH2015} proved that
\begin{equation}\label{Shaoeq}\Tr_d(\T)=(m-1)^{n-1} \sum_{F \in \F_d} \frac{b(F)}{c(F)} \pi_F(\T) |\W(F)|,
\end{equation}
where $\pi_F(\T)=\prod_{i=1}^d t_{i_j,\alpha_j}$ if $F=((i_1,\alpha_1),\ldots,(i_d,\alpha_d))$.

Let $\H$ be an $m$-uniform hypergraph on $n$ vertices and let $\A(\H)$ be the adjacency tensor of $\H$.
 Given an ordering of the vertices of $\mathcal{H}$, let
$$\mathcal{F}_{d}(\mathcal{H}):=\{(e_{1}(v_{1}),\ldots,e_{d}(v_{d})):e_{i}\in E(\mathcal{H}),v_{1}\leq\cdots\leq v_{d}\},$$
be the set of $d$-tuples of ordered rooted edges, where $e_{i}(v_{i})$ is an edge $e_{i}$ with root of $v_{i}\in e_{i}$ for $i\in [d]$. Define a rooted directed star $S_{e_{i}}(v_{i})=(e_{i},\{(v_{i},u):u\in e_{i}\backslash\{v_{i}\}\})$ for each $i\in [d]$,
and multi-directed graph $R(F)=\bigcup_{i=1}^{d}S_{e_{i}}(v_{i})$ associated with $F\in \mathcal{F}_{d}(\mathcal{H})$. Let
$$\mathcal{F}_{d}^{\epsilon}(\mathcal{H}):=\{F\in \mathcal{F}_{d}(\mathcal{H}):R(F)\hspace{0.4em}is\hspace{0.4em}Eulerian\}.$$
For an $F\in\mathcal{F}_{d}^{\epsilon}(\mathcal{H})$, denote $V(F):=V(R(F))$, $r_{v}(F)$ the number of edges in $F$ with $v$ as the root, and $d_{v}^{+}(F)=(m-1)r_{v}(F)$ (namely, the outdegree of $v$ in $R(F)$). Denote by $\tau(F):=\tau_{u}(R(F))$ the
number of \emph{arborescences} of $R(F)$ with root $u$ (namely, a directed $u$-rooted spanning tree such that all vertices except
$u$ has a directed path from itself to $u$), which is equal to the principal minor of the Laplacian matrix $L(R(F))$ of $R(F)$ by
deleting the row and column indexed by $u$ \cite{EB1987, TS1941}. As $R(F)$ is Eulerian, $\tau_{u}(R(F))$ is independent of the choice of the root
$u$ so that the root $u$ is omitted. Fan et al. \cite{FYa} give an expression of the $d$-th trace of $\mathcal{H}$ as follows.

\begin{lem}[\cite{FYa}]\label{TraF}
For an $m$-uniform hypergraph $\mathcal{H}$ on $n$ vertices,
$$Tr_{d}(\mathcal{H})=d(m-1)^{n}\sum_{F\in \mathcal{F}_{d}^{\epsilon}(\mathcal{H})}\frac{\tau(F)}{\prod_{v\in V(F)}d_{v}^{+}(F)}.$$
\end{lem}

For each $F\in \mathcal{F}_{d}^{\epsilon}(\mathcal{H})$, we get a multi-hypergraph induced by the edges in $F$ by omitting the roots, denoted by $\mathcal{V}_{F}$, which is an $m$-uniform and $m$-valent multi-hypergraph called {\it Veblen hypergraph} \cite{CC2021}. On the other side, given a Veblen hypergraph $H$, a rooting of $H$ is an ordering $F=(e_{1}(v_{1}),\ldots,e_{t}(v_{t}))$ of all edges of $H$, where $v_{i}$ is the root of $e_{i}$ for $i\in[t]$, and $v_{1}\leq\cdots\leq v_{t}$ under the given order of the vertices of $H$. If $R(F)$ is Eulerian, then $F$ is called an {\it Euler rooting} of $H$; in this case, $H$ is called {\it Euler rooted} with each edge rooted as in $F$ by omitting the order.
Note that an Euler rooted hypergraph may have more than one Euler rooting as a vertex can occur as roots of different edges.
Denote by $\mathscr{R}(H)$ the set of Euler rooting of $H$.

Denote by $\V_{d}(\mathcal{H})$ the set of connected Veblen hypergraphs with $d$ edges associated with $\mathcal{H}$ as follows:
$$\V_{d}(\mathcal{H})=\cup_{\G\in \CC(\mathcal{H})}\{\mathcal{V}_{F}: F\in \mathcal{F}_{d}^{\epsilon}(\mathcal{H}), \underline{\mathcal{V}_{F}}=\G\},$$
where $\CC(\mathcal{H})$ denotes the set of connected sub-hypergraphs of $\mathcal{H}$, and $\underline{H}$ denotes the underlying hypergraph of a multi-hypergraph $H$ which is obtained removing duplicate edges of $H$.

For a hypergraph $\G$, denote
$$\WW_d(\G)=\{\mathcal{V}_{F}: F\in \mathcal{F}_{d}^{\epsilon}(\G), \underline{\mathcal{V}_{F}}=G\},$$
that is, the set of Veblen hypergraph with $\G$ as underlying hypergraph which has an Euler rooting.
So,
$$\V_{d}(\mathcal{H})=\cup_{\G\in \CC(\mathcal{H})}\WW_{d}(\G).$$
For a Veblen hypergraph $H$,
denote
$$C_{H}=\sum_{F\in \R(H)}\frac{\tau(F)}{\prod_{v\in V(F)}d_{v}^{+}(F)},$$
and
\begin{equation}\label{sml}\tr_{d}(\G)=\sum_{H\in \WW_{d}(\G)}C_{H}.
\end{equation}

By Lemma \ref{TraF}, we get another expression of $\Tr_{d}(\mathcal{H})$ as follows:
\begin{equation}\label{newex2}
\begin{split}
\Tr_{d}(\mathcal{H})=&d(m-1)^{n}\sum_{H\in \V_d(\mathcal{H})}\sum_{F\in \R(H)}\frac{\tau(F)}{\prod_{v\in V(F)}d_{v}^{+}(F)}\\
=&d(m-1)^{n}\sum_{\G\in \CC(\mathcal{H})}\sum_{H\in \WW_{d}(G)}C_{H}\\
=&d(m-1)^{n}\sum_{\G\in \CC(\mathcal{H})}\tr_{d}(\G).
\end{split}
\end{equation}

\begin{lem}\label{newex}
Let $\H$ be an $m$-uniform hypergraph on $n$ vertices.
Then
$$\Tr_{d}(\mathcal{H})=d(m-1)^{n}\sum_{G\in \CC(\mathcal{H})}\tr_{d}(G).$$
\end{lem}

At the end of this section, we give an inequality involved with combinatorial numbers, which will be used in the later proofs.

\begin{lem}\label{combineq}
Let $x,y$ be positive integers, and $a,b$ be nonnegative integers. Then
\begin{equation}\label{ineq}(x+y+a)! b!+(x+y+b)! a! > (x+a)! (y+b)! + (x+b)! (y+a)!.\end{equation}
\end{lem}

\begin{proof}
By symmetry we can assume that $a \ge b$ and $x \ge y$.
Then
\begin{align*} &(x+y+a)! b!-(x+a)! (y+b)! +(x+y+b)! a!-(x+b)! (y+a)!\\
& =(x+a)!b!\left((x+a+1)\cdots (x+a+y)-(b+1)\cdots (b+y)\right)\\
& ~~~+
(x+b)!a! \left((x+b+1)\cdots (x+b+y)-(a+1)\cdots (a+y)\right).
\end{align*}
So, if $x+b \ge a$, then Eq. (\ref{ineq}) holds.
Otherwise, we have
\begin{align*} &(x+y+a)! b!-(x+a)! (y+b)! +(x+y+b)! a!-(x+b)! (y+a)!\\
& =(x+a)!b!\left((x+a+1)\cdots (x+a+y)-(b+1)\cdots (b+y)\right)\\
& ~~~-
(x+b)!a! \left((a+1)\cdots (a+y)-(x+b+1)\cdots (x+b+y)\right).
\end{align*}
It is easily seen that $(x+a)!b!\ge (x+b)!a!$ and
$$(x+a+1)\cdot (x+a+y)-(b+1)\cdot (b+y)>(a+1)\cdots (a+y)-(x+b+1)\cdots (x+b+y).$$
The result also follows.
\end{proof}

\section{Traces of hypertrees}
The \emph{coalescence} of two nontrivial connected hypergraphs $\H_1$ and $\H_2$ is the hypergraph obtained by identifying one vertex $v_1$ of $\H_1$ and one vertex $v_2$ of $\H_2$ to produce a new vertex $u$, denoted by $\H_1(v_1) \odot  \H_2(v_2)$, also written as $\H_1(u) \odot  \H_2(u)$.
Let $v'_1$ be another vertex of $\H_1$ different from $v_1$.
The coalescence $\H_1(v'_1) \odot  \H_2(v_2)$ is called obtained from $\H_1(v_1) \odot  \H_2(v_2)$ by \emph{relocating $\H_1$ from $v_1$ to $v'_1$}.
A vertex $u$ of a onnected hypergraph $\H$ is called a \emph{cut vertex} of $\H$  if $\H$ can be written as $\H_1(u) \odot  \H_2(u)$, where  $\H_1,\H_2$ are both nontrivial and connected, called the \emph{branches} of $\H$.

The following lemma gives a characterization of Veblen hypergraphs associated with a hypertree.

\begin{lem}[\cite{FYa}]\label{tree_root}
Let $H$ be an $m$-uniform Veblen multi-hypergraph whose underlying hypergraph $\underline{H}$ is a hypertree.
Then $H$ is uniquely Euler rooted such that all vertices of each edge occur as roots of the edge in a same number of times, and hence every edge of $H$ repeats in a multiple of $m$ times.
\end{lem}

Let $H$ be an $m$-uniform Veblen hypergraph with $d$ edges whose underlying hypergraph $\underline{H}=\breve{\T}$ is a hypertree.
By Lemma \ref{tree_root}, $m \mid d$, and
$H$ can be expressed as a weighted hypertree $\breve{\T}(\omega)$, where
$$ \omega: E(\breve{\T}) \to \Z^+,$$
such that the multiplicity of an edge $e \in E(H)$ is $m\omega(e)$, and $\omg(\breve{\T}):=\sum_{e \in E(\breve{\T})} \omega(e)=d/m$.
%, implying that $\hat{\T}$ contains at most $d/m$ edges.
So,
\begin{equation}\label{WdT} \WW_d(\breve{\T})=\{\breve{\T}(\omega): \omg(\breve{\T})=d/m\},
\end{equation}
and $ \WW_d(\breve{\T}) \ne \emptyset$ if and only if $m \mid d$.
Fan et al. \cite{FYh} proved that
\begin{equation}\label{SpeT}
 C_{\breve{\T}(\omega)}  = (m-1)^{-|V(\breve{\T})|}m^{(m-2)|E(\breve{\T})|} \left(\prod_{e \in E(\breve{\T})}\omg(e)\right)^{m-1} \prod_{v \in V(\breve{\T})} \frac{(d_v(\hat{\T}(\omega))-1)!}{\prod_{e \in E_v(\breve{\T})} \omega(e)!},
 \end{equation}
 where $d_v(\breve{\T}(\omega))=\sum_{e \in E_v(\breve{\T}) } \omega(e)$, the weighted degree of the vertex $v$ in $\breve{\T}(\omega)$.

Let $\T$ be an $m$-uniform hypertree.
By Lemma \ref{tree_root}, for $d \in \mathbb{Z}^+$, $\V_d(\T) \ne \emptyset$ if and only if $m \mid d$.
%So, in this subsection we always assume that \emph{$d$ is a positive multiple of $m$} when discussing $\Tr_d(\T)$; otherwise, $\Tr_d(\T)=0$ by Eq. (\ref{newex}).
By above discussion, we have
$$\V_d(\T)=\cup_{\breve{\T}\in \CC(\T)}\WW_d(\breve{\T})=
\cup_{\breve{\T}\in \CC(\T)}\{\breve{\T}(\omega): \omg(\breve{\T})=d/m\}.$$
 Note that $\prod_{v \in V(\breve{\T})}\prod_{e \in E(\breve{T})} \omg(e)! =\prod_{e \in E(\hat{\T})} (\omg(e)!)^m$.
By Lemma \ref{newex}, Eqs. (\ref{WdT}) and (\ref{SpeT}), we get an expression of the traces of the hypertrees.

\begin{thm}\label{TraTree}
Let $\T$ be an $m$-uniform hypertree.
If $m \mid d$, then
\begin{equation}
\Tr_{d}(\mathcal{T})=d(m-1)^{|V({\T})|}\sum_{\breve{\T}\in \CC(\T)}\tr_d(\breve{\T}), ~~~\tr_d(\breve{\T})=
\sum_{\omg: \omg(\breve{\T})=d/m}C_{\breve{\T}(\omega)},
\end{equation}
where
\begin{equation}\label{paraT}C_{\breve{\T}(\omega)}  =(m-1)^{-|V(\breve{\T})|}m^{(m-2)|E(\breve{\T})|}\prod_{v\in V(\breve{\T})} (d_{v}(\breve{\T}(\omega))-1)!\prod_{e\in E(\breve{\T})}\frac{\omega(e)^{m-1}}{(\omega(e)!)^{m}};
\end{equation}
otherwise, $\Tr_{d}(\mathcal{T})=\tr_d(\breve{\T})=0$.
\end{thm}

Let $\mathcal{H}_{0},\mathcal{H}_{1},\ldots,\mathcal{H}_{p}$ be pairwise disjoint connected hypergraphs, where $p\geq1$.
Let $v_{1},\ldots,v_{p}\in V(\mathcal{H}_{0})$, and $u_i \in V(\H_i)$ for $i \in [p]$. Denote by $\mathcal{H}_{0}(v_{1},\ldots,v_{p})\odot(\mathcal{H}_{1}(u_{1}),\ldots,\mathcal{H}_{p}(u_{p}))$ the hypergraph obtained
from $\mathcal{H}_{0}$ by attaching $\mathcal{H}_{1},\ldots,\mathcal{H}_{p}$ to $\mathcal{H}_{0}$ with $u_{i}$ identified
with $v_{i}$ for each $i\in [p]$, also written as $\mathcal{H}_{0}(v_{1},\ldots,v_{p})\odot(\mathcal{H}_{1}(v_{1}),\ldots,\mathcal{H}_{p}(v_{p}))$.
Let $\H$ be a hypergraph, and $\H_i$ be a sub-hypergraph of $\H$ for $i \in [s+t]$.
Denote by $\CC(\H;\H_1,\ldots,\H_s,{\H}^\times_{s+1},\ldots,{\H}^\times_{s+t})$ the set of connected sub-hypergraphs of $\H$ which contain the edges of $\H_1,\ldots,\H_s$ and contain no edges of ${\H}_{s+1},\ldots,{\H}_{s+t}$,
where $s,t$ are nonnegative integers.
Let $S \subseteq V(\H)$. Denote $\H-S$ the sub-hypergraph of $\H$ obtained by deleting the vertices $S$ together with the edges containing the vertices of $S$.

\begin{figure}[h]
\centering
\includegraphics[scale=.8]{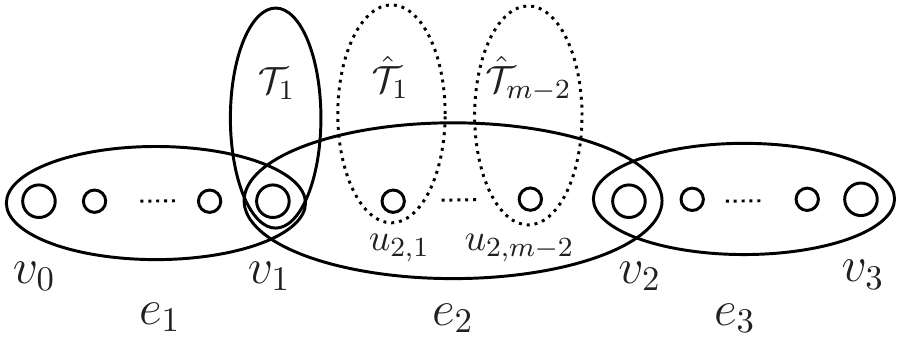}
\caption{The hypergraph $\H$ in Lemma \ref{Treepert}}\label{Hcoale}
\end{figure}

\begin{lem}\label{Treepert}
Let $P_3$ be a path on vertices $v_0,v_1,v_2,v_3$ with edges $\{v_{i-1},v_{i}\}$ for $i=1,2,3$.
Let $P_3^m$ be the power of $P_3$ with edges $e_i=\{v_{i-1},v_i ,u_{i,1},\ldots,u_{i,m-2}\}$ for $i=1,2,3$.
Let $\H=P_3^m(v_1,u_{2,1},\ldots,u_{2,m-2}) \odot ({\T}_1(v_1),\hat{\T}_1(u_{2,1}),\ldots,\hat{\T}_{m-2}(u_{2,m-2}))$,
where ${\T}_1$ is a nontrivial $m$-uniform hypertree, and $\hat{\T}_i$ is a $m$-uniform hypertree allowed to be trivial with only one vertex $u_{2,i}$ for $i \in [m-2]$; see Fig. \ref{Hcoale}.
Let $\T(w)$ be a nontrivial $m$-uniform hypertree with root $w$.
Then
\begin{equation}
\Tr_d(\H(v_1)\odot\T(w)) \ge \Tr_d(\H(v_2)\odot\T(w)),\end{equation}
with strict inequality if $m \mid d$ and $d/m \ge 2$.
\end{lem}

\begin{proof}
Let $\H_1:=\H(v_1)\odot\T(w)$ and $\H_2:=\H(v_2)\odot\T(w)$.
We first give decompositions of $\CC(\H_1)$ and $\CC(\H_2)$ as follows.
$$\CC(\H_i)=\CC(\H_i;\T^\times)\cup\CC(\H_i;\T,\T_1^\times)
\cup \CC(\H_i;\T,\T_1),i=1,2.$$
Furthermore, as all hypergraphs in $\CC(\H_i;\T,\T_1)$ are connected,
$$\CC(\H_1;\T,\T_1)=\CC(\H_1;\T,\T_1,e_2) \cup \CC(\H_1;\T,\T_1,e_2^\times),~
\CC(\H_2;\T,\T_1)=\CC(\H_2;\T,\T_1,e_2).$$
It is easily seen that $\CC(\H_1;\T^\times)=\CC(\H_2;\T^\times)$.
There is an isomorphism $\phi$ between $\H_1-(V(\T_1)\backslash\{v_1\})$ and  $\H_2-(V(\T_1)\backslash\{v_1\})$ which maps
 $v_0$ to $v_3$, $v_1$ to $v_2$.
  So, there is bijection from $\CC(\H_1;\T,\T_1^\times)$ to $\CC(\H_2;\T,\T_1^\times)$ such that $\G$ is mapped to  $\phi|_{\G}(\G)$, and $\G$ is isomorphic to $\phi|_{\G}(\G)$.

%\begin{figure}[h]
%\centering
%\includegraphics[scale=.9]{Hcoal2.pdf}
%\caption{The hypergraphs $\H(v_1)\odot\T(w)$ and $\H(v_2)\odot\T(w)$ in Lemma \ref{Treepert}}\label{Hcoale2}
%\end{figure}

By Lemma \ref{newex} and the above decompositions, we  have
\begin{equation}\label{tradif}\small
\Tr_d(\H_1)-\Tr_d(\H_2)=d(m-1)^{|V(\H_1)|} \left(\sum_{\G \in \CC(\H_1;\T,\T_1,e_2)}\tr_d(\G)-\sum_{\G \in \CC(\H_2;\T,\T_1,e_2)}\tr_d(\G)+\sum_{\G \in \CC(\H_1;\T,\T_1,e_2^\times)}\tr_d(\G)\right).
\end{equation}
We will prove that
\begin{equation}\label{ineg}
\sum_{\G \in \CC(\H_1;\T,\T_1,e_2)}\tr_d(\G)\ge \sum_{\G \in \CC(\H_2;\T,\T_1,e_2)}\tr_d(\G),
\end{equation}
with strict inequality if $m \mid d$  and $d/m \ge 3$.

By Lemma \ref{TraTree}, we assume that $m \mid d$; otherwise $\tr_d(\G)=0$.
Note that $\CC(\H_2;\T,\T_1,e_2)$ or $\CC(\H_1;\T,\T_1,e_2)$ is nonempty if and only if $d/m \ge 3$.
So we also assume that $d/m \ge 3$.
For each $\G \in \CC(\H_2;\T,\T_1,e_2)$, we can write
$\G =\H'(v_2)\odot \T'(w)$, where $\H'$ is a sub-hypergraph of $\H$ which contains
$e_2$ and a rooted sub-hypertree $\T'_1(v_1)$ of $\T_1$, and $\T'(w)$ is a rooted sub-hypertree of $\T$.
%Similarly, each graph $\tilde{G} \in \CC(\H_1;\T,\T_1,e_2^m)$ can be written as
%$\tilde{G}=\H'(v_1)\odot \T'(w)$;
There is a bijection $\psi$ between $\CC(\H_2;\T,\T_1,e_2)$ and $\CC(\H_1;\T,\T_1,e_2)$ such that $$\psi(\G)=\psi(\H'(v_2)\odot \T'(w))=\H'(v_1)\odot \T'(w)=:\tilde{\G}.$$
Also, each weight function $\omg: E(\G) \to \Z^+$ is naturally associated with a weight function $\tilde{\omg}: E(\tilde{\G}) \to \Z^+$ such that $\omg|_{E(\H')}=\tilde{\omg}|_{E(\H')}$ and  $\omg|_{E(\T'(w))}=\tilde{\omg}|_{E(\T'(w))}$,
and there is a bijection between $\WW_d(\G)$ and $\WW_d(\tilde{\G})$ such that $\G(\omg)$ is mapped to $\tilde{\G}(\tilde{\omg})$.

%In the following, we let $G=\H'(v_2)\odot \T'(w)$ and $\tilde{G}=\H'(v_1)\odot \T'(w)$, and let $\omg: E(G) \to \Z^+$ with $\omg(G)=d/m$, which induces a weight function $\tilde{\omg}: E(\tilde{G}) \to \Z^+$.
Let $d_{v_1}(\T'_1(\omg|_{E(\T'_1)})=x$ and $d_{w}(\T'(\omg|_{E(\T')})=y$.
We divide the discussion into cases.

Case 1: $e_1 \notin E(\G)$ and $e_3 \notin E(\G)$.
By Theorem \ref{TraTree},
$$C_{\G(\omg)}=(x+\omg(e_2)-1)!(y+\omg(e_2)-1)!f_{\G;v_1,v_2}(\omg),$$
$$C_{\tilde{\G}(\tilde{\omg})}=(x+y+\omg(e_2)-1)!(\omg(e_2)-1)!f_{\tilde{\G};v_1,v_2}(\tilde{\omg}),$$
where $f_{\G;v_1,v_2}(\omg)=f_{\tilde{\G};v_1,v_2}(\tilde{\omg})$ and
 \begin{equation}\label{fomg}
 f_{\G;v_1,v_2}(\omega):=(m-1)^{-|V(\G)|}m^{(m-2)|E(\G)|}\prod_{v\in V(\G)\backslash \{v_{1},v_{2}\}} (d_{v}(\G(\omega))-1)!\prod_{e\in E(\G)}\frac{\omega(e)^{m-1}}{(\omega(e)!)^{m}}.
 \end{equation}
So we have
$C_{\tilde{\G}(\tilde{\omg})}>C_{\G(\omg)}$, and
$\tr_d(\tilde{\G})>\tr_d(\G)$, which implies that
\begin{equation}\label{n1n3}
\sum_{\G \in \CC(\H_1;\T,\T_1,e_2,e_1^\times,e_3^\times )}\tr_d(\G)> \sum_{\G \in \CC(\H_2;\T,\T_1,e_2,e_1^\times,e_3^\times)}\tr_d(\G).
\end{equation}

Case 2: $e_1 \in E(\G)$ and $e_3 \notin E(\G)$, or $e_1 \notin E(\G)$ and $e_3 \in E(\G)$.
Let $\G[e_1] \in \CC(\H_2;\T,\T_1,e_2,e_1,e_3^\times)$.
  Let $\G[e_3] \in \CC(\H_2;\T,\T_1,e_2,e_1^\times, e_3)$, which is obtained from $\G[e_1]$ by relocating the edge $e_1$ from $v_1$ to $v_2$ and labelling the edge $e_1$ as $e_3$.
So, there is a bijection from $\CC(\H_2;\T,\T_1,e_2,e_1,e_3^\times)$ to $\CC(\H_2;\T,\T_1,e_2,e_1^\times, e_3)$ which maps $G[e_1]$ to
$G[e_3]$.
Also by the map $\psi$ defined before,
$\psi$ is a bijection between $\CC(\H_2;\T,\T_1,e_2,e_1,e_3^\times)$ and
$\CC(\H_1;\T,\T_1,e_2,e_1,e_3^\times)$ which maps $\G[e_1]$ to
$\tilde{\G}[e_1]:=\psi(\G[e_1])$, and also a  bijection between
$\CC(\H_2;\T,\T_1,e_2,e_1^\times, e_3)$ and
$\CC(\H_2;\T,\T_1,e_2,e_1^\times, e_3)$ which maps
$\G[e_3]$ to $\tilde{\G}[e_3]:=\psi(\G[e_3])$.

Each weight function $\omg: E(\G[e_1]) \to \mathbb{Z}^+$ induces a weight function $\omg': E(\G[e_3]) \to \mathbb{Z}^+$ such that $\omg'(e_3)=\omg(e_1)$ and $\omg'|_{E(\G[e_3])\backslash\{e_3\}}=\omg|_{E(\G[e_1])\backslash\{e_1\}}$.
As defined before, $\omg$ and $\omg'$ induce $\tilde{\omg}$ defined on $E(\tilde{\G}[e_1])$ and $\tilde{\omg'}$ defined on $E(\tilde{\G}[e_3])$.
By Lemma \ref{TraTree}, we have
$$C_{\G[e_1](\omg)}=(x+\omg(e_1)+\omg(e_2)-1)!(y+\omg(e_2)-1)!f_{\G[e_1];v_1,v_2}(\omg),$$
$$C_{\G[e_3](\omg')}=(x+\omg(e_2)-1)!(y+\omg(e_1)+\omg(e_2)-1)!f_{\G[e_3];v_1,v_2}(\omg'),$$
$$C_{\tilde{\G}[e_1](\tilde{\omg})}=(x+y+\omg(e_1)+\omg(e_2)-1)!(\omg(e_2)-1)!f_{\tilde{\G}[e_1];v_1,v_2}(\tilde{\omg}),$$
$$C_{\tilde{\G}[e_3](\tilde{\omg'})}=(x+y+\omg(e_2)-1)!(\omg(e_1)+\omg(e_2)-1)!f_{\tilde{\G}[e_3];v_1,v_2}(\tilde{\omg'}),$$
where $f_{\G[e_1^m];v_1,v_2}(\omg)=f_{\G[e_3^m];v_1,v_2}(\omg')
=f_{\tilde{\G}[e_1];v_1,v_2}(\tilde{\omg})=f_{\tilde{\G}[e_1];v_1,v_2}(\tilde{\omg'})$ as defined in Eq. (\ref{fomg}).
By Lemma \ref{combineq},
$$C_{\tilde{\G}[e_1](\tilde{\omg})}+C_{\tilde{\G}[e_3](\tilde{\omg'})}>C_{\G[e_1](\omg)}+C_{\G[e_3](\omg')}.$$
So
$$ \tr_d(\tilde{\G}[e_1])+\tr_d(\tilde{\G}[e_3])>\tr_d(\G[e_1])+\tr_d(\G[e_3]),$$
which implies that
\begin{equation}\label{n1orn3}\small
\sum_{\tilde{\G} \in \CC(\H_1;\T,\T_1,e_2,e_1,e_3^\times ) \cup \CC(\H_1;\T,\T_1,e_2,e_1^\times,e_3)
}\tr_d(\tilde{\G})
> \sum_{\G \in \CC(\H_2;\T,\T_1,e_2,e_1,e_3^\times) \cup \CC(\H_2;\T,\T_1,e_2,e_1^\times,e_3)}\tr_d(\G).
\end{equation}

Case 3: $e_1 \in E(\G)$ and $e_3 \in E(\G)$.
As noted before, each hypergraph $\G \in \CC(\H_2;\T,\T_1,e_2,e_1,e_3)$ is bijectively corresponding to
$\tilde{\G}:=\psi(\G) \in \CC(\H_1;\T,\T_1,e_2,e_1,e_3)$ and each weight function $\omg$ on $E(\G)$ induces a weight function $\tilde{\omg}$ on $E(\tilde{\G})$.
Also, each $\omg$ (respectively, $\tilde{\omg}$) is associated with another weight function
$\omg'$ (respectively, $\tilde{\omg}'$) only by swapping the weight of $e_1$ and the weight of $e_3$.
By Lemma \ref{TraTree}, we have
$$ C_{\G(\omg)}=(x+\omg(e_1)+\omg(e_2)-1)!(y+\omg(e_2)+\omg(e_3)-1)!f_{\G;v_1,v_2}(\omg),$$
$$ C_{\G(\omg')}=(x+\omg(e_2)+\omg(e_3)-1)!(y+\omg(e_1)+\omg(e_2)-1)!f_{\G;v_1,v_2}(\omg'),$$
$$ C_{\tilde{\G}(\tilde{\omg})}=(x+y+\omg(e_1)+\omg(e_2)-1)!(\omg(e_2)+\omg(e_3)-1)!f_{\tilde{\G};v_1,v_2}(\tilde{\omg}),$$
$$ C_{\tilde{\G}(\tilde{\omg}')}=(x+\omg(e_2)+\omg(e_3)-1)!(y+\omg(e_1)+\omg(e_2)-1)!f_{\tilde{\G};v_1,v_2}(\tilde{\omg}'),$$
where $f_{\G;v_1,v_2}(\omg)=f_{\G;v_1,v_2}(\omg')=f_{\tilde{\G};v_1,v_2}(\tilde{\omg})
=f_{\tilde{\G};v_1,v_2}(\tilde{\omg}')$ as defined in Eq. (\ref{fomg}).
By Lemma \ref{combineq},
$$C_{\tilde{\G}(\tilde{\omg})}+C_{\tilde{\G}(\tilde{\omg'})}>C_{G(\omg)}+C_{G(\omg')}.$$
So
\begin{equation}\label{symsum}
\begin{split}
\tr_d(\tilde{\G})-\tr_d(\G)&=
\sum_{\omg: \omg(\G)=d/m} \left(C_{\tilde{\G}(\tilde{\omg})}-C_{\G(\omg)}\right)\\
&=
\frac{1}{2}\sum_{\omg: \omg(G)=d/m} \left((C_{\tilde{\G}(\tilde{\omg})}+C_{\tilde{\G}(\tilde{\omg}')})-(C_{\G(\omg)}+C_{\G(\omg')})\right)\\
&>0,\end{split}
\end{equation}
which implies that
\begin{equation}\label{1and3}
\sum_{\tilde{\G} \in \CC(\H_1;\T,\T_1,e_2,e_1,e_3)
}\tr_d(\tilde{\G})
> \sum_{\G \in \CC(\H_2;\T,\T_1,e_2,e_1,e_3)}\tr_d(\G).
\end{equation}

Combining Eqs. (\ref{n1n3}), (\ref{n1orn3}) and (\ref{1and3}), we arrive at the inequality in Eq. (\ref{ineg}).
Note that $\tr_d(\G) \ge 0$ for each $\G \in \CC(\H_1;\T,\T_1,e_2^\times)$, with strict inequality if and only if $m \mid d$  and $d/m \ge 2$ as in this case $\CC(\H_1;\T,\T_1,e_2^\times) \ne \emptyset$ by Lemma \ref{tree_root}.
So the result follows by Eq. (\ref{tradif}).
\end{proof}

We note that the method to prove the inequality (\ref{symsum}) may be called ``\emph{symmetric sum}'', which will be used in later discussion, e.g. Lemma \ref{Treepert2} and Lemma \ref{cycleper}.

\section{Traces of linear unicyclic hypergraphs}

In this section we will discuss the traces of linear unicyclic hypergraphs.
We need the following lemma to characterize the Veblen hypergraph containing cycles.

\begin{lem}[\cite{FYa}]\label{core}
Let $H$ be an $m$-uniform Veblen multi-hypergraph, and let $e$ be an edge of $\underline{H}$  which contains a cored vertex.
If $H$ has an Euler rooting, then $e$ repeats $k \cdot m$ times for some positive integer $k$, and all cored vertices in $e$ occur as a root of $e$ in $k$ times.
\end{lem}

Let $\U$ be an $m$-uniform linear unicyclic hypergraph.
As $\U$ is linear, $\U$ contains a cycle $C_n^m$ of length $n \ge 3$, and $U$ is obtained from $C_n^m$ by some hypertrees to the vertices of $C_n^m$, namely $\U=C_n^m(v_1,\ldots,v_p) \odot (\T_1(v_1),\ldots,\T_p(v_p))$, where $\T_i$'s are hypertrees for $i \in [p]$ and $p$ is a nonnegative integer.

We have a decomposition of the Veblen hypergraph associated with $\U$:
\begin{equation}\label{decom} \V_d(\U)=\V_d(\U,[\hat{C_n^m}]) \cup \V_d(\U,[C_n^m]),
\end{equation}
where $\mathcal{V}_d(U^m;[\hat{C}_n^m])$ (respectively, $\mathcal{V}_d(U^m;[C_n^m])$ ) denotes the subset of  $\mathcal{V}_d(U^m)$ which consists of Veblen hypergraphs that contain no $C_n^m$ (respectively, contain  $C_n^m$).

For each $H \in \V_d(\U,[\hat{C_n^m}])$, $\underline{H}$ is a hypertree.
By Lemma \ref{tree_root}, $\V_d(\U,[\hat{C_n^m}]) \ne \emptyset$ if and only if $m \mid d$; and in this case
\begin{equation}\label{Utree}\V_d(\U,[\hat{C_n^m}])=\{\breve{\T}(\omg): \breve{\T} \in \mathcal{C}_{tree}^*(\U), \omg(\breve{\T})=d/m\},
\end{equation}
where $\mathcal{C}_{tree}^*(\U)$ denoted the connected sub-hypergraphs of $\U$ which are hypertrees.

For each $H \in \V_d(\U,[C_n^m])$,  $\underline{H}$ contains $C_n^m$.
 We may assume that $\underline{H}=C_n^m(v_1,\ldots,v_q) \odot (\T'_1(v_1),\ldots,\T'_q(v_q))$, where $\T'_i(v_i)$ is a subhypertree of $\T_i(v_i)$ with root $v_i$ for $i \in [q]$ and $q \le p$.
As each $v_i$ is a cut vertex, $H|_{\T'_i(v_i)}$, the limitation of $H$ on the vertex set of $\T'_i(v_i)$, is still a Veblen hypergraph (hypertree).
By Lemma \ref{tree_root}, the number of edges of $H|_{\T'_i(v_i)}$, denoted by $d_i$, is a multiple of $m$, and
$H|_{\T'_i(v_i)}=\T'_i(v_i)(\omg^i)$ with $\omg^i(\T'_i(v_i))=d_i/m$.
So, $H|_{C_n^m}$ is a Veblen hypergraph.
 By Lemma \ref{core}, as each edge of $C_n^m$ contains cored vertices,
 the number of edges of $H|_{C_n^m}$,  denoted by $d_0$, is a multiple of $m$, and
 $H|_{C_n^m}=C_n^m(\omg^0)$ with $\omg^0(C_n^m)=d_0/m$.
 So $H$ is a weighted hypergraph $\underline{H}(\omg)$ with weight
$$\omg: E(\underline{H}) \to \Z^+$$
such that each edge $e$ of $H$ has multiplicity $m \omg(e)$, and $\omg(H)=d/m$.
By above discussion, we find that $\V_d(\U,[C_n^m])\ne \emptyset$ if and only if $m \mid d$; and in this case
\begin{equation}\label{Ucycle}\V_d(\U,[{C_n^m}])=\{\mathcal{G}(\omg): \mathcal{G} \in \mathcal{C}_{cycle}^*(\U), \omg(\mathcal{G})=d/m\},
\end{equation}
and for $\mathcal{G} \in \mathcal{C}_{cycle}^*(\U)$,
\begin{equation}\label{Ucycle1}
\WW_d(\G)=\{\mathcal{G}(\omg): \omg(\mathcal{G})=d/m\},
\end{equation}
where $\mathcal{C}_{cycle}^*(\U)$ denots the connected sub-hypergraphs of $\U$ which contains the cycle $C_n^m$.

\begin{cor}\label{Udeco1}
Let $\U$ be a linear unicyclic hypergraph which contains a cycle $C_n^m$.
Then
\begin{equation}\label{Udeco}
\V_d(\U)=\V_d(\U,[\hat{C_n^m}]) \cup \V_d(\U,[C_n^m]),
\end{equation}
and each of $\V_d(\U,[\hat{C_n^m}])$ and $\V_d(\U,[C_n^m])$ is nonempty if and only if $m \mid d$,
where $\V_d(\U,[\hat{C_n^m}])$ and $\V_d(\U,[C_n^m])$ are defined in Eqs. (\ref{Utree}) and (\ref{Ucycle}) respectively.
\end{cor}

 Let $\mathcal{G}(\omg) \in \V_d(\U,[C_n^m])$ with $\mathcal{G}=C_n^m(v_1,\ldots,v_q) \odot (\T'_1(v_1),\ldots,\T'_q(v_q))$.
As discussed above,
$$\mathcal{G}(\omg)=C_n^m(\omg^0)(v_1,\ldots,v_q) \odot (\T'_1(\omg^1)(v_1),\ldots,\T'_q(\omg^q)(v_q)).$$
Note that $\mathcal{G}(\omg)$ is Euler rooted if and only if $C_n^m(\omg^0)$ and all of
$\T'_1(\omg^1),\ldots,\T'_q(\omg^q)$ are Euler rooted.

Suppose that $C_n$ has vertices $u_1,\ldots, u_n$ and edges $\{u_i,u_{i+1}\}$ for $i \in [n]$, $u_{n+1}=u_1$.
We label the edges of $C_n^m$ as $e_i=\{u_i,u_{i+1},w_{i1}, \ldots, w_{i,m-2}\}$ for $i \in [n]$.
Denote $\omg^0_i=\omg^0(e_i)$ for $i \in [n]$, and $\omg^0_{\min}=\min_{i \in [n]}\omg^0(e_i)$.
As shown in \cite{FYh},
the set of Euler rootings of $C_n^m(\omg^0)$ has a decomposition:
 $$\R(C_n^m(\omg^0))=\cup_{x=0}^{2\omega^0_{\min}}\R(C_n^m(\omg^0);x),$$
 where $\R(C_n^m(\omg^0);x)$ consists of those rootings $F$ such that $u_{i}$ acts a root of $e_{i}$ (respectively,  $e_{i-1}$) in $\omega^0_{i}-\omega^0_{\min}+x$ (respectively,  $\omega^0_{i-1}-\omega^0_{min}+x$) times for $i\in [n]$, and each vertex of  $e_i \backslash \{u_i,u_{i+1}\}$ acts as a root of $e_i$ in $\omega(e_i)$ times.
 So, $r_v(F)=d_v(C_n^m(\omg^0))$ for each $F \in \R(C_n^m(\omg^0);x)$.
Define $\Omega_{C_n^m}(\omg^0)$ as follows for the later discussion:
\begin{equation}\label{fonC}
\Omega_{C_n^m}(\omg^0)=\sum_{x=0}^{2\omega^0_{\min}} \prod_{i=1}^{n}\frac{(\omega^0_{i}!)^2}{(\omega^0_{i-1}+\omega^0_{\min}-x)!
(\omega^0_{i}-\omega^0_{\min}+x)!}
\sum_{l=0}^{n-1}\prod_{i=1}^{l}(\omega^0_{i}+\omega^0_{\min}-x)
\prod_{i=l+2}^{n}(\omega^0_{i}-\omega^0_{\min}+x).
\end{equation}

Consequently, we have a decompositon
$$ \R(\mathcal{G}(\omg))=\cup_{x=0}^{2\omega^0_{\min}}\R(\mathcal{G}(\omg);x),$$
where
$$\R(\mathcal{G}(\omg);x)=\{(F_0,F_1,\ldots,F_q): F_0 \in \R(C_n^m(\omg^0);x), F_i \in \R(\T'_i(\omg^i)),  i \in [q]\}.$$
By the formula given in \cite{FYh}, for each $(F_0,F_1,\ldots,F_q) \in \R(\mathcal{G}(\omg);x)$,
$$\tau(F_0)=2 m^{n(m-2)-1} \left(\prod_{i=1}^n \omg^0_i\right)^{m-2} \sum_{l=0}^{n-1} \prod_{i=1}^l (\omg_i^0+\omg_{\min}^0-x) \prod_{i=l+2}^n(\omg_i^0-\omg_{\min}^0+x),$$
$$\tau(F_i)=m^{(m-2)|E(\T'_i)|}\left(\prod_{e \in E(\T'_i)}\omg^i(e)\right)^{m-1}, i \in [q],$$
and hence $\tau(F)=\tau(F_0) \prod_{i=1}^q \tau(F_i)$, which equals
\begin{equation}
\tau(F)=2 m^{(m-2)|E(\mathcal{G})|-1}
\left(\prod_{e \in E(\mathcal{G})}\omg(e)\right)^{m-1}
\frac{\sum_{l=0}^{n-1} \prod_{i=1}^l (\omg_i^0+\omg_{\min}^0-x) \prod_{i=l+2}^n(\omg_i^0-\omg_{\min}^0+x)}{\left(\prod_{e \in E(C_n^m)}\omg(e)\right)^{m-1}}.
\end{equation}

Note that by Lemma \ref{tree_root}, for each $F_i \in \R(\T'_i(\omg^i))$, every vertex $v \in e \in \T'_i(\omg^i)$ acts as root of $e$ in $\omg^i(e)$ times.
 So, we have
%\begin{align*}|\R(\mathcal{G}(\omg);x)|&=\prod_{i=1}^q {r_{v_i}(F_0)+r_{v_i}(F_i) \choose r_{v_i}(F_0)} \cdot |\R(C_n^m(\omg^0))| \cdot \prod_{i=1}^ q |\R(\T'_i(\omg^i))|\\
%&=\prod_{i=1}^q \frac{(d_{v_i}(C_n^m(\omg^0))+ d_{v_i}(\T'_i(\omg^i)))!}{d_{v_i}(C_n^m(\omg^0))!d_{v_i}(\T'_i(\omg^i))!} \cdot \prod_{i=1}^n {r_{u_i}(F_0) \choose r_{u_i}(e_i^m), r_{u_i}(e_{i-1}^m)} \cdot \prod_{i=1}^q \prod_{v \in V(\T'_i)} {r_v(F_i) \choose (r_v(e): v \in e)}\\
%&=\prod_{i=1}^q \frac{d_{v_i}(\mathcal{G}(\omg))!}{d_{v_i}(C_n^m(\omg^0))!d_{v_i}(\T'_i(\omg^i))!} \cdot \prod_{i=1}^n \frac{(\omg^0_i+\omg^0_{i-1})!}{(\omg_i^0+\omg_{\min}^0-x)!(\omg_i^0-\omg_{\min}^0+x)!} \cdot \prod_{i=1}^q \prod_{v \in V(\T'_i)} \frac{d_v(\T'_i(\omg^i))!}{\prod_{e \in E_v(\T'_i)}\omg^i(e)!}\\
%\end{align*}
$$|\R(\mathcal{G}(\omg);x)|=\prod_{v \in V(\mathcal{G})} {r_{v}(F) \choose r_v(e): e \in E_v(\mathcal{G})}
 =\prod_{v \in V(\mathcal{G})} \frac{r_{v}(F)!}{\prod_{e \in E_v(\mathcal{G})}  r_v(e)!}
=\frac{\prod_{v \in V(\mathcal{G})} d_v(\mathcal{G}(\omg))!}
{\prod_{v\in V(\mathcal{G})} \prod_{e \in E_v(\mathcal{G})}  r_v(e)!}.
$$
As
\begin{align*}
\prod_{v \in V(\mathcal{G})} d_v(\mathcal{G}(\omg))!&=
\prod_{e \in E_v(\mathcal{G})}   \prod_{v\in e} r_v(e)!
= \prod_{e \in \cup_{i=1}^q E(\T'_i)}  \prod_{v\in e} r_v(e)! \cdot \prod_{e \in E(C_n^m)}  \prod_{v\in e} r_v(e)!\\
&=\prod_{e \in \cup_{i=1}^q E(\T'_i)}  (\omg(e)!)^m \cdot
\prod_{e \in E(C_n^m)}(\omg(e)!)^{m-2} \cdot \prod_{i=1}^n (\omg_i^0+\omg_{\min}^0-x)!(\omg_i^0-\omg_{\min}^0+x)!\\
&=\prod_{e \in E(\mathcal{G})}  (\omg(e)!)^m \cdot \frac{\prod_{i=1}^n (\omg_i^0+\omg_{\min}^0-x)!(\omg_i^0-\omg_{\min}^0+x)!}{\prod_{i=1}^n (\omg^0_i!)^2},
\end{align*}
we have
\begin{equation}
|\R(\mathcal{G}(\omg);x)|=\frac{d_v(\mathcal{G}(\omg))!\prod_{i=1}^n (\omg^0_i!)^2}{\prod_{e \in E(\mathcal{G})}  (\omg(e)!)^m \prod_{i=1}^n (\omg_i^0+\omg_{\min}^0-x)!(\omg_i^0-\omg_{\min}^0+x)!}.
\end{equation}

Note that for each $F \in \R(\mathcal{G}(\omega);x)$, $$\prod_{v\in V(F)}d_{v}^{+}(F)=(m-1)^{|V(\mathcal{G})|}\prod_{v\in V(\mathcal{G})}d_{v}(\mathcal{G}(\omega)).$$
By definition, we have
\begin{equation}\label{paraC}
\begin{split}
C_{\mathcal{G}(\omega)}=&\sum_{F\in \R(\mathcal{G}(\omega))}\frac{\tau(F)}{\prod_{v\in V(F)}d_{v}^{+}(F)}=\sum_{x=0}^{2\omega_{\min}}\sum_{F\in \R(\mathcal{G}(\omega);x)}\frac{\tau(F)}{\prod_{v\in V(F)}d_{v}^{+}(F)}\\
&=\sum_{x=0}^{2\omega_{\min}} |\R(\mathcal{G}(\omg);x)| \cdot \frac{\tau(F)}{\prod_{v\in V(F)}d_{v}^{+}(F)}\\
=&2(m-1)^{-|V(\mathcal{G})|}m^{(m-2)|E(\mathcal{G})|-1}\prod_{v\in V(\mathcal{G})} (d_{v}(\mathcal{G}(\omega))-1)!\prod_{e\in E(\mathcal{G})}\frac{\omega(e)^{m-1}}{(\omega(e)!)^{m}} \cdot \Omega_{C_n^m}(\omg^0).
%\sum_{x=0}^{2\omega_{\min}} \prod_{i=1}^{n}\frac{(\omega^0_{i}!)^2}{(\omega^0_{i-1}+\omega^0_{min}-x)!
%(\omega^0_{i}-\omega^0_{min}+x)!}
%\sum_{l=0}^{n-1}\prod_{i=1}^{l}(\omega^0_{i}+\omega^0_{min}-x)
%\prod_{i=l+2}^{n}(\omega^0_{i}-\omega^0_{min}+x).
\end{split}
\end{equation}

We are now giving an expression for the traces of linear unicyclic hypergraps.

\begin{thm}\label{traU}
Let $\U$ be an $m$-uniform linear unicyclic hypergraph.
If $m \mid d$, then
\begin{equation}
\Tr_d(\U)=d(m-1)^{|V(\U)|} \left(\sum_{\breve{\T} \in \CC_{tree}(\U)}\tr_d(\breve{\T})+\sum_{\mathcal{G} \in \CC_{cycle}(\U)}\tr_d(\mathcal{G})\right),
\end{equation}
and
\begin{equation}
\tr_d(\breve{\T})=\sum_{\omg: \omg(\breve{\T})=d/m} C_{\breve{\T}(\omg)},~~~
\tr_d(\mathcal{G})=\sum_{\omg: \omg(\mathcal{G})=d/m} C_{\mathcal{G}(\omg)},
\end{equation}
where $C_{\breve{\T}(\omg)}$ and $C_{\mathcal{G}(\omg)}$ are defined in Eqs. (\ref{paraT}) and (\ref{paraC}) respectively;
otherwise, $\Tr_d(\U)=0$.
\end{thm}

\begin{proof}
By Eq. (\ref{newex2}) and the decomposition (\ref{Udeco}), we have
\begin{align*}
\Tr_d(\U)&=d(m-1)^{|V(\U)|}\sum_{H \in \V_d(\U)} C_H\\
&=d(m-1)^{|V(\U)|} \left( \sum_{H \in \V_d(\U,[\hat{C_n^m}])} C_H +
\sum_{H \in \V_d(\U,[C_n^m])} C_H\right).
\end{align*}
By Corollary \ref{Udeco1}, each of $\V_d(\U,[\hat{C_n^m}])$ and $\V_d(\U,[C_n^m])$ is nonempty if and only if $m \mid d$.
So we assume $m \mid d$ in the following.
By Eq. (\ref{Utree}),
$$\V_d(\U,[\hat{C_n^m}])=\cup_{\hat{\T} \in \CC_{tree}(\U)}\WW_d(\breve{\T}),$$
where $\WW_d(\breve{\T})=\{\breve{\T}(\omg): \omg(\breve{\T})=d/m\}$.
So
$$\sum_{H \in \V_d(\U,[\hat{C_n^m}])} C_H=\sum_{\breve{\T} \in \CC_{tree}(\U)} \sum_{H \in \WW_d(\breve{\T})} C_H=\sum_{\breve{\T} \in \CC_{tree}(\U)} \tr_d(\breve{\T}).$$
By the definition of $\WW_d(\breve{\T})$ and $\tr_d(\breve{\T})$, we have
$$\tr_d(\breve{\T})=\sum_{\omg: \omg(\breve{\T})=d/m} C_{\breve{\T}(\omg)}.$$

Similarly, by Eqs. (\ref{Ucycle}) and (\ref{Ucycle1}), we have
$\V_d(\U,[C_n^m])=\cup_{\G \in \CC_{cycle}(\U)}\WW_d(\G)$
and $\WW_d(\G)=\{\mathcal{G}(\omg): \omg(\mathcal{G})=d/m\}$.
So
$$\sum_{H \in \V_d(\U,[C_n^m])} C_H=\sum_{\G \in \CC_{cycle}(\U)} \sum_{H \in \WW_d(\G)} C_H=\sum_{\G \in \CC_{cycle}(\U)} \tr_d(\G),$$
and $$\tr_d(\G)=\sum_{\omg: \omg(\G)=d/m} C_{\G(\omg)}.$$
The result follows.
\end{proof}

\begin{lem}\label{attStar}
Let $\U$ be an $m$-uniform linear unicyclic hypergraph which contains a cycle $C_n^m$, and let
$e$ be an edge of $C_n^m$ which contains a vertex $v$ of degree $2$ and a vertex $u$ with degree $1$.
Let $\T$ be an $m$-uniform nontrivial hypertree with root $w$.
Then
$$\Tr_d(\U(v)\odot \T(w)) \ge \Tr_d(\U(u)\odot \T(w))$$
with equality if  $d/m$ is an integer and $d/m \ge 2$,
\end{lem}

\begin{proof}
Let $\H_1:=\U(v)\odot \T(w)$ and $\H_2:=\U(u)\odot \T(w)$.
We have
$$\CC(\H_i)=\CC(\H_i;\T^\times)\cup\CC(\H_i;\T,(C_n^m)^\times)
\cup \CC(\H_i;\T,C_n^m),i=1,2.$$
Furthermore, as all hypergraphs in $\CC(\H_i;\T,\T_1)$ are connected,
$$\CC(\H_1;\T,C_n^m)=\CC(\H_1;\T,C_n^m,e) \cup \CC(\H_1;\T,C_n^m,e^\times),~
\CC(\H_2;\T,C_n^m)=\CC(\H_2;\T,C_n^m,e).$$
It is easily seen that $\CC(\H_1;\T^\times)=\CC(\H_2;\T^\times)$
and $\CC(\H_1;\T,(C_n^m)^\times)=\CC(\H_2;\T,(C_n^m)^\times)$.
So,
\begin{equation}\label{tradif3}\small
\Tr_d(\H_1)-\Tr_d(\H_2)=d(m-1)^n \left(\sum_{\G \in \CC(\H_1;\T,C_n^m,e)}\tr_d(\G)-\sum_{\G \in \CC(\H_2;\T,C_n^m,e)}\tr_d(\G)+\sum_{\G \in \CC(\H_1;\T,C_n^m,e^\times)}\tr_d(\G)\right).
\end{equation}
We will prove that
\begin{equation}\label{ineg3}
\sum_{\G \in \CC(\H_1;\T,C_n^m,e)}\tr_d(\G)\ge \sum_{\G \in \CC(\H_2;\T,C_n^m,e)}\tr_d(\G),
\end{equation}
with strict inequality if $d/m$ is an integer and $d/m \ge 3$.

By Theorem \ref{traU}, we assume that $m \mid d$; otherwise $\tr_d(\G)=0$.
Note that $\CC(\H_2;\T,C_n^m,e)$ or $\CC(\H_1;\T,C_n^m,e)$ is nonempty if and only if $d/m \ge 2$.
So we also assume that $d/m \ge 2$.
For each $\G \in \CC(\H_2;\T,C_n^m,e)$, we can write
$\G =\U'(u)\odot \T'(w)$, where $\U'$ is a sub-hypergraph of $\U$ which contains the edge $e$, and $\T'(w)$ is a rooted sub-hypertree of $\T$.
%Similarly, each graph $\tilde{G} \in \CC(\H_1;\T,\T_1,e_2^m)$ can be written as
%$\tilde{G}=\H'(v_1)\odot \T'(w)$;
There is a bijection $\psi$ between $\CC(\H_2;\T,C_n^m,e)$ and $\CC(\H_1;\T,C_n^m,e)$ such that $\psi(\G)=\psi(\U'(u)\odot \T'(w))=\U'(v)\odot \T'(w)=:\tilde{\G}.$
Also, each weight function $\omg: E(\G) \to \Z^+$ is naturally associated with a weight function $\tilde{\omg}: E(\tilde{\G}) \to \Z^+$ such that $\omg|_{E(\U')}=\tilde{\omg}|_{E(\U')}$ and  $\omg|_{E(\T'(w))}=\tilde{\omg}|_{E(\T'(w))}$.

%Now let $G=\U'(u)\odot \T'(w)$ and $\tilde{G}=\psi(G)=\U'(v)\odot \T'(w)$,
%and let $\omg: E(G) \to \mathbb{Z}^+$ with $\omg(G)=d/m$ which associates with $\tilde{\omg}: E(\tilde{G}) \to \mathbb{Z}^+$ defined as above.
Let $d_{w}(\T'(\omg|_{E(\T')})=x$. We divide the discussions into cases.

Case 1. $\U'$ is a hypertree. Then $\U'$ is obtained from $e$ by attaching some hypertrees to its vertices, especially attaching $\hat{\T}_1$ to $v$.
Let $d_{v}(\hat{\T}_1(\omg|_{E(\hat{\T}_1)})=y$.
If $\hat{\T}_1$ is trivial (or equivalently $y=0$), then $\G$ is isomorphic to $\tilde{\G}$.
So we assume that $y>0$.
By Theorem \ref{traU}, we have
$$ C_{\G(\omg)}=(x+\omg(e)-1)!(y+\omg(e)-1)!f_{\G;u,v}(\omg),$$
$$ C_{\tilde{\G}(\tilde{\omg})}=(x+y+\omg(e)-1)!(\omg(e)-1)!f_{\G;u,v}(\tilde{\omg}),$$
where $f_{\G;u,v}(\tilde{\omg})=f_{\G;u,v}(\omg)$ as defined in (\ref{fomg}).
It is easily seen $C_{\tilde{\G}(\tilde{\omg})}>C_{\G(\omg)}$, and
$\tr_d(\tilde{\G})>\tr_d(\G)$ in this case.
So we have
\begin{equation}\label{ineg4}
\sum_{\tilde{\G} \in \CC(\H_1;\T,C_n^m,e)\cap \CC_{tree}(\H_1)}\tr_d(\tilde{\G})\ge \sum_{\G \in \CC(\H_2;\T,C_n^m,e)\cap\CC_{tree}(\H_2)}\tr_d(\G),
\end{equation}
with strictly inequality if $d/m$ is an integer and $d/m \ge 3$, as in this case $\U'$ can have another edge of $C_n^m$ that contains $v$.

Case 2. $\U'$ contains $C_n^m$. Let $d_v(\U'-\{e\})(\omg|_{E(\U')\backslash\{e\}})=y$, where $\U'-\{e\}$ is obtained from $\U'$ by deleting the edge $e$.
By Theorem \ref{traU}, we have
$$ C_{\G(\omg)}=(x+\omg(e)-1)!(y+\omg(e)-1)!h_{\G;u,v}(\omg),$$
$$ C_{\tilde{\G}(\tilde{\omg})}=(x+y+\omg(e)-1)!(\omg(e)-1)!h_{\G;u,v}(\tilde{\omg}),$$
where $h_{\G;u,v}(\tilde{\omg})=h_{\G;u,v}(\omg)$ defined as follows:
\begin{equation}\label{Lomg}
h_{\G;u,v}(\omg):=2(m-1)^{-|V(\G)|}m^{(m-2)|E(\G)|-1}\prod_{v\in V(\G)\backslash \{v,u\}} (d_{v}(\G(\omega))-1)!\prod_{e\in E(\G)}\frac{\omega(e)^{m-1}}{(\omega(e)!)^{m}}\cdot \Omega_{C_n^m}(\omg|_{E(C_n^m)}).
\end{equation}
It easily seen $C_{\tilde{\G}(\tilde{\omg})}>C_{\G(\omg)}$, and
$\tr_d(\tilde{\G})>\tr_d(\G)$ in this case.
So we have
\begin{equation}\label{ineg5}
\sum_{\tilde{\G} \in \CC(\H_1;\T,C_n^m,e)\cap \CC_{cycle}(\H_1)}\tr_d(\tilde{\G})\ge \sum_{\G \in \CC(\H_2;\T,C_n^m,e)\cap\CC_{cycle}(\H_2)}\tr_d(\G),
\end{equation}
with strictly inequality if $m \mid d$ is an integer and $d/m \ge n+1$.

By Eqs. (\ref{ineg4}) and (\ref{ineg5}), we arrive the inequality (\ref{ineg3}).
As $\CC(\H_1;\T,C_n^m,e^\times)\ne \emptyset $ if $m \mid d$ and $d/m \ge 2$, we get the desired result by Eq. (\ref{tradif3}).
\end{proof}

\section{Maximum Estrada index Of hypertrees with perfect matching}
In this section, we will determine the hypertree(s) with maximum Estrada index among all hypertrees with given number of edges and perfect matchings.
We need the following lemma.

\begin{lem}[\cite{FYa}]\label{edgePer}
Let $e=\{u,v_{1},\ldots,v_{m-1}\}$ be an edge, and let $\mathcal{H}_{1}=e(v_{1},\ldots,v_{p})\odot(\G_{1}(\tilde{v}_{1}),\ldots,\G_{p}(\widetilde{v}_{p}))$, where $1\leq p\leq m-1$.
Let $\H_2(w)$ be a nontrivial $m$-uniform rooted hypergraph. Then
$$\Tr_{d}(\mathcal{H}_{1}(v_{1})\odot\mathcal{H}_{2}(w))\geq \Tr_{d}(\mathcal{H}_{1}(u)\odot\mathcal{H}_{2}(w)),$$
with strict inequality if $m|d$ and $d/m\geq2$, and hence
$$EE(\mathcal{H}_{1}(v_{1})\odot\mathcal{H}_{2}(w))> EE(\mathcal{H}_{1}(u)\odot\mathcal{H}_{2}(w)).$$
\end{lem}

\begin{figure}
\centering
\includegraphics[scale=.9]{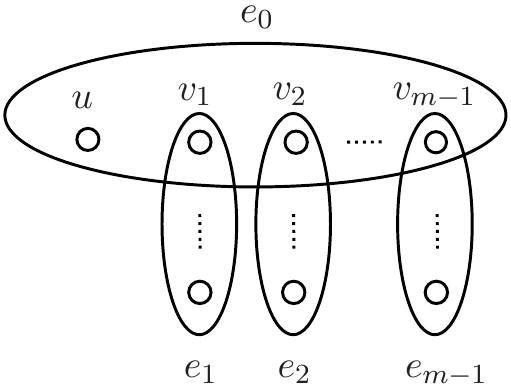}
\caption{An $m$-comb $Comb_u$}\label{comb}
\end{figure}

We first introduce a special $m$-uniform hypertree.
Let $e_0,e_1,\ldots,e_{m-1}$ be $m$ disjoint edges ($m$-sets), where  $e_{0}=\{u,v_{1},\ldots,v_{m-1}\}$, and $u_i \in e_i$ for $i \in [m-1]$.
Denote $Comb_{u}=e_{0}(v_{1},\ldots,v_{m-1})\odot(e_{1}(u_{1}),\ldots,e_{m-1}(u_{m-1}))$,
called an \emph{$m$-comb} with endpoint $u$; see Fig. \ref{comb}.
Let $e$ be an edge ($m$-set) contains a vertex $v$.
Let $Comb_{u_1},\ldots,Comb_{u_t}$ be $t$ pairwise disjoint $m$-combs.
Denote $\T_{m,t}=e(v)\odot(Comb_{u_{1}}(u_{1}),\ldots,Comb_{u_{t}}(u_{t}))$; see Fig. \ref{Tcomb}.
It is easy to see that $\T_{m,t}$ has a unique perfect matching of size $t(m-1)+1$.

\begin{figure}
\centering
\includegraphics[scale=1]{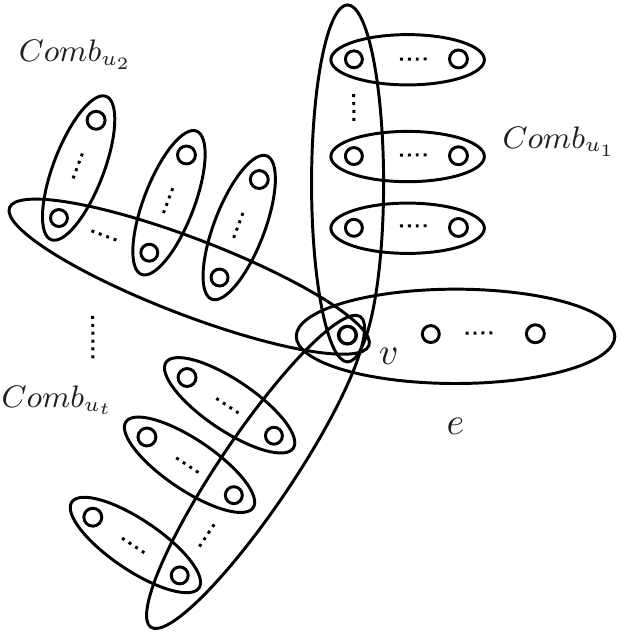}
\caption{The hypergraph $\T_{m,t}$}\label{Tcomb}
\end{figure}

Let $\T$ be an $m$-uniform hypertree of order $n$.
If $\T$ has a perfect matching $M$ of size $k$, then $n=km$, and $|E(\T)|=\frac{n-1}{m-1}=\frac{km-1}{m-1}=k+\frac{k-1}{m-1}$, which implies that
$(k-1) \mid (m-1)$, and $\T$ has $\frac{k-1}{m-1}$ edges outside $M$.

\begin{lem}
Let $\T$ be an $m$-uniform hypertree with perfect matching.
If $\T$ has more than one edge, then
$$\T=\T'(u) \odot Comb_u(u),$$
where $\T'$ is a sub-hypertree of $\T$ with perfect matching;
consequently, $\T$ has a unique perfect matching.
\end{lem}

\begin{proof}
Let $P_d^m$ be a longest path of $\T$ with consecutive edges $e_1,\ldots,e_d$, where $d \ge 2$.
Let $e_{d-1}=\{u,v_1,\ldots, v_{m-1}\}$, where $v_{m-1} \in e_d$.
Clearly, $e_d$ is a pendent edge of $\T$ by the definition of $P_d^m$.
Let $M$ be a perfect matching of $\T$.
Surely, $e_d \in M$, which implies that $v_{m-1}$ is covered by $M$.
So $e_{d-1} \notin M$, and each vertex $v_i$ of $e_{d-1}$ will be covered by some edge $f_i \in M$ for $i \in [m-2]$.
As the maximum length of the paths of $\T$ is $d$, all edges $f_i$ for $i \in [m-2]$
are pendent edges.
Similarly, except $e_{d-1}$ and $f_i$, no edges contain $v_i$ for $i \in [m-2]$, and
 except $e_{d-1}$ and $e_d$, no edges contain $v_{m-1}$.
 So we get a comb $$Comb_u=e_{d-1}(v_1,\ldots,v_{m-2},v_{m-1}) \odot (f_1(v_1),\ldots,f_{m-2}(v_{m-2}), e_d(v_{m-1})),$$
 and $\T=\T'(u) \odot Comb_u(u)$.
As all vertices of $Comb_u(u)$ are covered by the edges of $M$ except $u$,
$\T'$ has a perfect matching.
By induction, $M$ is the unique perfect matching of $\T$.
\end{proof}

\begin{thm}
Let $\T$ be an $m$-uniform hypertree of order $mk$ with perfect matching. Then
$$EE(\T)\leq EE(\T_{m,\frac{k-1}{m-1}}),$$
with equality if and only if $\T=\T_{m,\frac{k-1}{m-1}}$.
\end{thm}

\begin{proof}
Let $\T_{0}$ be a hypertree with maximum Estrada index among all hypertrees of order $mk$ with perfect matchings.
If $k=1$, then $\T_0$ is an edge, and the result follows clearly.
So we assume that $k>1$.
Let $M$ be the unique perfect matching of $\T_{0}$.
 We assert that if $e\in M$, then $e$ is a pendent edge of $\T_{0}$;
 otherwise, letting $e=\{u,v_1,\ldots,v_{m-1}\}$, then
 $\T_0=\H(u) \odot \T_0'(u)$, where $\H=e(v_1,\ldots,v_p)\odot ({\T}'_1,\ldots,{\T}'_p)$, $1 \le p \le m-1$.
By Lemma \ref{edgePer}, we have
$$EE(\T_{0})<EE(\mathcal{H}(v_{1})\odot\T_0'(u)).$$
As $\mathcal{H}(v_{1})\odot\T'_0(u)$ also has a perfect matching,
we get a contradiction to the definition of $\T_{0}$.

Let $w$ be a vertex of $\T_0$ with maximum degree $\Delta>1$.
Then $\T_0$ contains a hyperstar $S_\Delta^m$ centered at $w$ as sub-hypertree.
Let $e_1,\ldots,e_\Delta$ be edges of $S_\Delta^m$ which share a common vertex $u$.
Suppose that $u$ is covered by $e_1 \in M$.
Then neither of $e_2,\ldots,e_\Delta$ belongs to $M$.
So every vertices of $e_i$ except $u$ are covered by the edges of $M$ for $i=2,\ldots,\Delta$, implying that $\T_0$ contains a comb  $\T_{m,\Delta-1}$.

If $\Delta=2$, then $\T_0=\T_{m,1}$, the result follow.
Assume that $\Delta>2$ and $\T_0 \ne \T_{m,\Delta-1}$.
By the above assertion, all edges of $M$ are pendent, so $\T_0=\T_{m,\Delta-1}(w_1,w_2,\ldots,w_t)\odot ({\T}'_1(w_1),{\T}'_2(w_2),\ldots,{\T}'_p(w_p))$,
where $w_1,\ldots,w_p$ are the vertices of $S_\Delta^m$ except $w$, and $ p \ge 1$.
Let $\tilde{\T}_0=\T_{m,\Delta-1}(w,w_2,\ldots,w_t)\odot ({\T}'_1(w_1),{\T}'_2(w_2),\ldots,{\T}'_p(w_p))$, which also has a perfect matching.
By Lemma \ref{Treepert},
we have
$$\Tr_d(\tilde{\T}_0)\ge \Tr_d(\T_0),$$
with strict inequality if $d/m$ is an integer and $d/m \ge 2$.
So, $$EE(\tilde{\T}_0)> EE(\T_0),$$
a contradiction to the definition of $\T_0$.
So, $\T_0 = \T_{m,\Delta-1}$ and $\Delta=\frac{k-1}{m-1}+1$.
\end{proof}

\section{Maximum Estrada index of linear unicyclic hypergraphs with given girth}
Recall the \emph{girth} of a hypergraph $\H$, denoted by $g(\H)$, is the minimum length of the cycles of $\H$. If $\H$ contains no cycles, then we define $g(\H) =+\infty$.
Let $\mathscr{U}_{z,g}^m$ denote the set of $m$-uniform linear unicyclic hypergraphs with $z$ edges and girth $g$.
In this section we will characterize the unicyclic hypergraphs with maximum Estrada index among all
hypergraphs in  $\mathscr{U}_{z,g}^m$.

\begin{lem}\label{UatStar}
Let $\U$ be $m$-uniform linear unicyclic hypergraph with maximum Estrada index among all hypergraphs in $\mathscr{U}_{z,g}^m$.
Then $\U$ is obtained from $C_g^m$ by attaching some hyperstars to the vertices of $C_g^m$ of degree $2$ with their centers identified with the vertices of $C_g^m$.
\end{lem}

\begin{proof}
It is known $\U$ is obtained from $C_g^m$ by attaching some hypertrees at the vertices of $C_g^m$.
Let $S=\{v_1,\ldots,v_g\}$ be the set of vertices of $C_g^m$ with degree $2$.
We first assert that all hypertrees are attached to the vertices of $S$.
Otherwise, let $\T$ be a hypertree attached at a cored vertex $w$ of $C_g^m$, where $w \notin S$, and let $e$ be an edge of $C_g^m$ which contains $w$ and a vertex $v \in S$.
Then $\U=\U'(w)\odot \T(w)$, where $\U'$ is a unicylic sub-hypergraph of $\U$.
By Lemma \ref{attStar},
$$\Tr_d(\U) \le \Tr_d(\U'(v)\odot \T(w)),$$
with strict inequality if $m \mid d$ and $d/m\ge 2$.
So $EE(\U)<EE(\U'(v)\odot \T(w))$, a contradiction to the definition of $\U$.

By the above discussion, $\U$ is obtained by attaching some hypertrees to the vertices of $S$.
We next assert all hypertrees are hyperstars centered at the vertices of $S$.
Otherwise, let $\T$ be a hypertree attached at $v \in S$, which  is not a hyperstar centered at $v$.
Then there exists an edge $e$ outside $C_g^m$ which contains the vertex $v$ and a vertex $w$ to which a sub-hypertree $\T'$ of $\T$ is attached.
Letting $e=\{v,w,u_1,\ldots,u_{m-2}\}$, we can write
$\U=e(v,w,u_1,\ldots,u_p)\odot (\U'(v),\T'(w),\hat{\T}_1(u_1),\ldots,\hat{\T}_p(u_p))$,
where $\U'$ is a unicylic sub-hypergraph of $\U$, and $\hat{\T}_i$ is a sub-hypertree of $\T$ for $i=1,\ldots,p$, $0 \le p \le m-2$. (Note that  $\hat{\T}_i(u_i)$ may be trivial.)
Let $\tilde{\U}$ be obtained from $\U$ by relocating $\T'$ from $w$ to $v$.
By Lemma \ref{edgePer}, we have $EE(\U) < EE(\tilde{\U})$; a contradiction.
The result follows.
\end{proof}

We will determine the unique $m$-uniform linear unicyclic hypergraph of girth $3$ with maximum Estrada index in $\mathscr{U}_{Z,3}^m$.
We need some lemmas for preparation.

\begin{lem}\label{Onedge}
Let $e=\{u,v_{1},\ldots,v_{m-1}\}$ be an edge, and let $\mathcal{H}=e(v_{1},\ldots,v_{p})\odot(\T_{1}(v_{1}),\ldots,\T_{p}(v_{p}))$, where $1\leq p\leq m-1$.
Let $\T$ be a nontrivial $m$-uniform hypertree with root $w$.

(1) Let  $\mathcal{H}_{1}=\mathcal{H}(v_{1})\odot \T(w)$ and $\mathcal{H}_{2}=\mathcal{H}(u)\odot \T(w)$.
Then
$$tr_{d}(\H_1)\geq tr_{d}(\H_2),$$
with strict inequality if $m|d$ and $d/m \ge |E(\H_1)|$.

(2) Let $\H_{11}=\H_1(v_1)\odot \tilde{\T}(\tilde{w})$ and $\H_{12}=\H_1(u)\odot \tilde{\T}(\tilde{w})$,
$\H_{21}=\H_2(v_1)\odot \tilde{\T}(\tilde{w})$ and $\H_{12}=\H_2(u)\odot \tilde{\T}(\tilde{w})$.
Then
$$\tr_d(\H_{11})+\tr_d(\H_{12})\ge \tr_d(\H_{21})+\tr_d(\H_{22}),$$
with strict inequality if $m|d$ and $d/m \ge |E(\H_{11})|$.
\end{lem}

\begin{proof}
(1) By Theorem \ref{TraTree}, if $m \mid d$, then
$$\tr_d(\H_i)=\sum_{\omg:\omg(\H_i)=d/m}C_{\H_i(\omg)}, i=1,2;$$
otherwise, $\tr_d(\H_i)=0$.
So we assume that $m \mid d$ in the following.
Note that each weight function $\omg: E(\H_1) \to \Z^+$ with $\omg(\H_1)=d/m$ induces a weight function
$\tilde{\omg}: E(\H_2) \to \Z^+$ such that $\tilde{\omg}|_{E(\H)}= \omg|_{E(\H)}$ and $\tilde{\omg}|_{E(\T(w))}= \omg|_{E(\T(w))}$.
Let $d_{v_1}(\T(\omg|_{E(\T_1)}))=x$, $d_w(\T(\omg|_{E(\T(w))}))=y$.
By Theorem \ref{TraTree},
$$C_{\H_1(\omg)}=(x+y+\omg(e)-1)!(\omg(e)-1)!f_{\H_1;v_1,u}(\omg),$$
$$C_{\H_2(\tilde{\omg})}=(x+\omg(e)-1)!(y+\omg(e)-1)!f_{\H_2;v_1,u}(\tilde{\omg}),$$
where $f_{\H_1;v_1,u}(\omg)=f_{\H_2;v_1,u}(\tilde{\omg})$ as defined in (\ref{fomg}).
Surely, $C_{\H_1(\omg)}>C_{\H_2(\tilde{\omg})}$ if $d/m \ge |E(\H_1)|$ as in this case $\H_i(\omg)$ exists for $i=1,2$.

(2) As discussed in (1), we assume that $m \mid d$.
Each weight function $\omg_{11}: E(\H_{11}) \to \Z^+$ with $\omg(\H_{11})=d/m$ induces a weight function $\omg_{ij}: E(\H_{ij}) \to \Z^+$ such that $\omg_{ij}|_{E(\H)}= \omg|_{E(\H)}$ and $\omg_{ij}|_{E(\T(w))}= \omg|_{E(\T(w))}$, and $\omg_{ij}|_{E(\T'(w'))}= \omg|_{E(\T'(w'))}$ for
$i,j=1,2$.
Let $d_{v_1}(\T(\omg|_{E(\T_1)}))=x$, $d_w(\T(\omg|_{E(\T(w))}))=y$, and $d_{\tilde{w}}(\T(\omg|_{E(\tilde{\T}(\tilde{w}))}))=z$.
We have
$$C_{\H_{11}(\omg_{11})}=(x+y+z+\omg_{11}(e)-1)!(\omg_{11}(e)-1)!f_{\H_{11};v_1,u}(\omg_{11}),$$
$$C_{\H_{12}(\omg_{12})}=(x+y+\omg_{11}(e)-1)!(z+\omg_{11}(e)-1)!f_{\H_{12};v_1,u}(\omg_{12}),$$
$$C_{\H_{21}(\omg_{21})}=(x+z+\omg_{11}(e)-1)!(y+\omg_{11}(e)-1)!f_{\H_{21};v_1,u}(\omg_{21}),$$
$$C_{\H_{22}(\omg_{22})}=(x+\omg_{11}(e)-1)!(y+z+\omg_{11}(e)-1)!f_{\H_{22};v_1,u}(\omg_{22}),$$
where $f_{\H_{11};v_1,u}(\omg_{11})=f_{\H_{21};v_1,u}(\omg_{21})=f_{\H_{12};v_1,u}(\omg_{12})
=f_{\H_{22};v_1,u}(\omg_{22})$ as defined in (\ref{fomg}).
By Lemma \ref{combineq},
$$C_{\H_{11}(\omg_{11})}+C_{\H_{12}(\omg_{21})}>C_{\H_{21}(\omg_{12})}+C_{\H_{22}(\omg_{22})},$$
which yields the result by Theorem \ref{TraTree}.
\end{proof}

\begin{figure}[h]
\centering
\includegraphics[scale=.9]{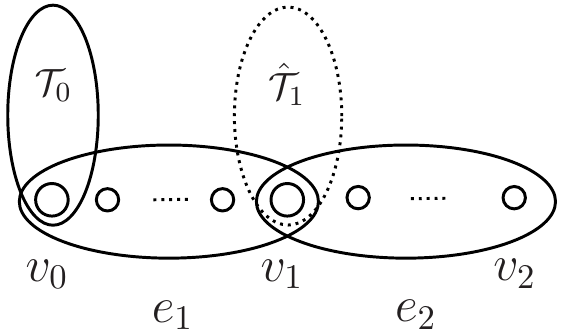}
\caption{The hypergraph $\H$ in Lemma \ref{Treepert2}}\label{Hcoalep2}
\end{figure}

\begin{lem}\label{Treepert2}
Let $P_2$ be a path on vertices $v_0,v_1,v_2$ with edges $\{v_{i-1},v_{i}\}$ for $i=1,2$.
Let $P_2^m$ be the power of $P_2$ with edges $e_i=\{v_{i-1},v_i, u_{i,1},\ldots,u_{i,m-2}\}$ for $i=1,2$.
Let $\H=P_2^m(v_0,v_1) \odot ({\T}_0(v_0),\hat{\T}_1(v_1))$,
where ${\T}_0$ is a nontrivial $m$-uniform hypertree, and $\hat{\T}_1$ is a $m$-uniform hypertree allowed to be trivial with only one vertex $v_{1}$; see Fig. \ref{Hcoalep2}.
Let $\T(w)$ be a $m$-uniform nontrivial hypertree with root $w$.
Then
$$
\tr_d(\H(v_0)\odot\T(w)) \ge \tr_d(\H(v_2)\odot\T(w)),$$
with strict inequality if $d/m$ is an integer and $d/m \ge 4$.
\end{lem}

\begin{proof}
Let $\H_1:=\H(v_0)\odot\T(w)$ and $\H_2:=\H(v_2)\odot\T(w)$.
By Lemma \ref{TraTree}, we assume that $m \mid d$; otherwise $\tr_d(\H_i)=0$.
Each weight function $\omg: E(\H_2) \to \mathbb{Z}^+$ with $\omg(\H_2)=d/m$ induces a weight function $\tilde{\omg}: E(\H_1) \to \mathbb{Z}^+$ such that
$\tilde{\omg}|_{E(\H)}=\omg|_{E(\H)}$ and  $\tilde{\omg}|_{E(\T(w))}=\omg|_{E(\T(w))}$.
Let $\omg'$ (respectively, $\tilde{\omg}'$) be obtained from $\omg$ (respectively, $\tilde{\omg}$) only by swapping the weight of $e_1$ and the weight of $e_2$.
Let $d_{v_0}(\T_0(\omg|_{E(\T_0)})=x$ and $d_{w}(\T(\omg|_{E(\T)})=y$.
By Theorem \ref{TraTree}, we have
$$ C_{\H_2(\omg)}=(x+\omg(e_1)-1)!(y+\omg(e_2)-1)!f_{\H_2;v_0,v_2}(\omg),$$
$$ C_{\H_2(\omg')}=(x+\omg(e_2)-1)!(y+\omg(e_1)-1)!f_{\H_2;v_0,v_2}(\omg'),$$
$$ C_{\H_1(\tilde{\omg})}=(x+y+\omg(e_1)-1)!(\omg(e_2)-1)!f_{\H_1;v_0,v_2}(\tilde{\omg}),$$
$$ C_{\H_1(\tilde{\omg}')}=(x+y+\omg(e_2)-1)!(y+\omg(e_1)-1)!f_{\H_1;v_0,v_2}(\tilde{\omg}'),$$
where $f_{\H_1;v_0,v_2}(\omg)=f_{\H_1;v_0,v_2}(\omg')=f_{\H_1;v_0,v_2}(\tilde{\omg})=f_{\H_1;v_0,v_2}(\tilde{\omg}')$ as defined in Eq. (\ref{fomg}).
By Lemma \ref{combineq},
$$C_{\H_1(\tilde{\omg})}+C_{\H_1(\tilde{\omg'})}>C_{\H_2(\omg)}+C_{\H_2(\omg')}.$$
So
\begin{align*}
\tr_d(\H_1)-\tr_d(\H_2)&=
\sum_{\omg: \omg(\H_2)=d/m} \left(C_{\H_1(\tilde{\omg})}-C_{\H_2(\omg)}\right)\\
&=
\frac{1}{2}\sum_{\omg: \omg(\H_2)=d/m} \left((C_{\H_1(\tilde{\omg})}+C_{\H_1(\tilde{\omg}')})-(C_{\H_2(\omg)}+C_{\H_2(\omg')})\right)\\
&>0.\end{align*}
The result follows.
\end{proof}

\begin{lem}\label{cycleper}
Let $C_3^m$ be the $m$-th power of $C_3$  with edges $e_i \supseteq \{v_i,v_{i+1}\}$ for $i=1,2,3$, where $v_4=v_1$.
Let $\H=C_3^m(v_1,v_3)\odot (\T_1(v_1),\hat{\T}_3(v_3))$, where $\T_1$ is nontrivial and $\hat{\T}_3$ may be trivial.
Let $\T_2$ be a nontrivial $m$-uniform hypertree with root $w$.
Then
$$\tr_d(\H(v_1)\odot \T_2(w)) \ge \tr_d(\H(v_2)\odot \T_2(w)),$$
with strict inequality if $m|d$ and $d/m \geq |E(\H(v_1)\odot \T_2(w))|.$
\end{lem}

\begin{proof}
Let $\H_1=\H(v_1)\odot \T_2(w)$ and $\H_2=\H(v_2)\odot \T_2(w)$.
By Theorem \ref{traU}, we assume $m \mid d$; otherwise $\tr_d(\H_i)=0$ for $i=1,2$.
Each weight function $\omg: E(\H_1)\to \Z^+$ induces a weight function $\tilde{\omg}: E(\H_2)\to \Z^+$ such that $\tilde{\omg}|_{E(\H)}=\omg|_{E(\H)}$ and $\tilde{\omg}|_{E(\T_2(w))}=\omg|_{E(\T_2(w))}$.
Let $\omg'$ (respectively  $\tilde{\omg}'$) be obtained from $\omg$ (respectively  $\tilde{\omg}$) by swapping the weight of $e_2$ and the weight of $e_3$.
Let $d_{v_1}(\T_1(\omg_{E(\T_1)}))=x$, $d_{w}(\T_2(\omg_{E(\T_2)}))=y$, and let $\omg_i:=\omg(e_i)$ for $i=1,2,3$.
By Theorem \ref{traU}, we get
$$C_{\mathcal{H}_{1}(\omg)}=(x+y+\omega_{1}+\omega_{3}-1)!(\omega_{1}+\omega_{2}-1)!h_{\H_1;v_1,v_2}(\omega),$$
$$C_{\mathcal{H}_{1}(\omg')}=(x+y+\omega_{1}+\omega_{2}-1)!(\omega_{1}+\omega_{3}-1)!h_{\H_1;v_1,v_2}(\omega'),$$
$$C_{\mathcal{H}_{2}(\tilde{\omg})}=(x+\omega_{1}+\omega_{3}-1)!(y+\omega_{1}+\omega_{2}-1)!h_{\H_2;v_1,v_2}(\tilde{\omg}),$$
$$C_{\mathcal{H}_{2}(\tilde{\omg}')}=(x+\omega_{1}+\omega_{2}-1)!(y+\omega_{1}+\omega_{3}-1)!h_{\H_2;v_1,v_2}(\tilde{\omg}'),$$
where $h_{\H_1;v_1,v_2}(\omega)=h_{\H_1;v_1,v_2}(\omega')=h_{\H_2;v_1,v_2}(\tilde{\omg})=h_{\H_2;v_1,v_2}(\tilde{\omg}')$ as defined in (\ref{Lomg}) by verifying $\Omega_{C_3^m}(\omg|_{E(C_3^m)})=\Omega_{C_3^m}(\omg'|_{E(C_3^m)})$.
By Lemma \ref{combineq}, we have
$$C_{\mathcal{H}_{1}(\omg)}+C_{\mathcal{H}_{1}(\omg')}>C_{\mathcal{H}_{2}(\tilde{\omg})}+C_{\mathcal{H}_{2}(\tilde{\omg}')}.$$
So
\begin{align*}
\tr_d(\H_1)-\tr_d(\H_2)&=
\sum_{\omg: \omg(G)=d/m} \left(C_{\H_1(\omg)}-C_{\H_2(\tilde{\omg})}\right)\\
&=
\frac{1}{2}\sum_{\omg: \omg(G)=d/m}
\left(
(C_{\H_1(\omg)}+C_{\H_1(\omg')})
-(C_{\H_2(\tilde{\omg})}+C_{\H_2(\tilde{\omg}')})
\right)\\
&>0.\end{align*}
The result follows.
\end{proof}

\begin{cor}\label{cycleTra}
Let $C_3^m$ be the $m$-th power of $C_3$  with edges $e_i \supseteq \{v_i,v_{i+1}\}$ for $i=1,2,3$, where $v_4=v_1$.
Let $\H=C_3^m(v_1,v_3)\odot (\T_1(v_1),\hat{\T}_3(v_3))$, where $\T_1$ is nontrivial and $\hat{\T}_3$ may be trivial.
Let $\T_2$ be a nontrivial $m$-uniform hypertree with root $w$.
Then
$$\Tr_d(\H(v_1)\odot \T_2(w)) \ge \Tr_d(\H(v_2)\odot \T_2(w)),$$
with strict inequality if $m|d$ and $d/m \geq 2.$
\end{cor}

\begin{proof}
Let $\H_1:=\H(v_1)\odot \T_2(w)$ and $\H_2:=\H(v_2)\odot \T_2(w)$.
We have
$$\CC(\H_i)=\CC(\H_i;\T_2^\times)
\cup\CC(\H_i;\T_2,\T_1^\times)\cup\CC(\H_i;\T_2,\T_1),i=1,2.$$
As all hypergraphs in $\CC(\H_i;\T_2,\T_1)$ are connected, we also have
$$ \CC(\H_1;\T_2,\T_1)=\CC(\H_1;\T_2,\T_1,e_1^m)\cup \CC(\H_1;\T_2,\T_1,(e_1^m)^\times),
\CC(\H_2;\T_2,\T_1)=\CC(\H_2;\T_2,\T_1,e_1^m).$$
Obviously, $\CC(\H_1;\T_2^\times)=\CC(\H_2;\T_2^\times)$, and there is a bijection $\phi$ from
$\CC(\H_1;\T_2,\T_1^\times)$ to $\CC(\H_2;\T_2,\T_1^\times)$
such that $G$ is isomorphic to $\phi(G)$ for each $G \in \CC(\H_1;\T_2,\T_1^\times)$.
By Theorem \ref{traU},
\begin{equation}\label{tradif4}\small
\Tr_d(\H_1)-\Tr_d(\H_2)=d(m-1)^N \left(\sum_{\G \in \CC(\H_1;\T_2,\T_1,e_1)}\tr_d(\G)-\sum_{\G \in \CC(\H_2;\T_2,\T_1,e_1)}\tr_d(\G)+\sum_{\G \in \CC(\H_1;\T_2,\T_1,e_1^\times)}\tr_d(\G)\right),
\end{equation}
where $N$ is the number of vertices of $\H_1$ or $\H_2$.
We will prove that
\begin{equation}\label{ineg4}
\sum_{\G \in \CC(\H_1;\T_2,\T_1,e_1)}\tr_d(\G)\ge \sum_{\G \in \CC(\H_2;\T_2,\T_1,e_1)}\tr_d(\G),
\end{equation}
with strict inequality if $d/m$ is an integer and $d/m \ge 3$.

By Theorem \ref{traU}, we may assume $m \mid d$; otherwise $\tr_d(\G)=0$.
Note that $\CC(\H_2;\T_2,\T_1,e_1^m)$ or $\CC(\H_1;\T_2,\T_1,e_1^m)$ is nonempty if and only if $d/m \ge 3$.
So we also assume that $d/m \ge 3$.
For each $\G \in \CC(\H_2;\T_2,\T_1,e_1^m)$, we can write $\G=\H'(v_2) \odot \T'(w)$, where $\H'$ is a sub-hypergraph of $\H$ which contains rooted sub-hypertree $\T'_1(v_1)$ of $\T_1(v_1)$ and the edge $e_1^m$, and $\T_2'(w)$ is a rooted sub-hypertree of $\T_2(w)$.
Then there is a bijection $\psi$ from $\CC(\H_2;\T_2,\T_1,e_1^m)$ to $\CC(\H_1;\T_2,\T_1,e_1^m)$ such that $\psi(\H'(v_2) \odot \T_2'(w))=\H'(v_1) \odot \T_2'(w)$.
Now we let $\G=\H'(v_2) \odot \T_2'(w)$ and $\tilde{\G}:=\psi(\G)=\H'(v_1) \odot \T_2'(w)$, and  divide the discussion into cases.
%Also, each weight function $\omg: E(\H'(v_2) \odot \T'(w))\to \Z^+$ induces a weight function $\tilde{\omg}: E(\H'(v_1) \odot \T'(w))\to \Z^+$ such that
%$\tilde{\omg}|_{E(\H')}=\omg|_{E(\H')}$ and $\tilde{\omg}|_{E(\T'(w))}=\omg|_{E(\T'(w))}$.

Case 1. $\G$ contains neither $e_2$ nor $e_3$.
Then $\tilde{\G},\G$ are respectively corresponding the hypergraphs $\H_1,\H_2$ in Lemma \ref{Onedge}(1), and hence
$\tr_d(\tilde{\G}) \ge \tr_d(\G).$
So we have
\begin{equation}\label{3cylce1}
\sum_{\tilde{\G} \in \CC(\H_1;\T_2,\T_1,e_1,e_2^\times,e_3^\times)}\tr_d(\tilde{\G})
\ge \sum_{\G \in \CC(\H_2;\T_2,\T_1,e_1^m,e_2^\times,e_3^\times)}\tr_d(\G),
\end{equation}
with strict inequality if $d/m \ge 3$ as in this case $\CC(\H_1;\T_2,\T_1,e_1^m,e_2^\times,e_3^\times)$ is nonempty.

Case 2. $\G$ contains exactly one edge of $e_2$ and $e_3$.
Let $\G[e_3] \in \CC(\H_2;\T_2,\T_1,e_1,e_2^\times, e_3)$ which contains $e_3$ but no $e_2$.
Let $\T'$ be the (edge maximal) sub-hypertree of $\G[e_3]$ which contains $e_3$ and contains no edges of $\T_1,\T_2$ or $e_1$.
Let $\G[e_2] \in \CC(\H_2;\T_2,\T_1,e_1,e_2,e_3^\times)$ which is obtained from $\G[e_2]$ by relocating the sub-hypertree $\T'$ attached at $v_1$ to $v_2$ and labeling the edge $e_3$ as $e_2$.
By the bijection $\psi$ defined before, we have
$\tilde{\G}[e_3] \in \CC(\H_1;\T_2,\T_1,e_1,e_2^\times, e_3)$ and
$\tilde{\G}[e_2] \in \CC(\H_1;\T_2,\T_1,e_1,e_2,e_3^\times)$.
Now $\tilde{\G}[e_3], \tilde{\G}[e_2],\G[e_3],\G[e_2]$ are corresponding $\H_{11},\H_{12},\H_{21},\H_{22}$ in Lemma \ref{Onedge}(2), and hence
$$ \tr_d(\tilde{\G}[e_3])+\tr_d(\tilde{\G}[e_2]) \ge \tr_d(\G[e_3])+\tr_d(\G[e_2]).$$
So we have
\begin{equation}\label{3cylce2}\small
\sum_{\tilde{\G} \in \CC(\H_1;\T_2,\T_1,e_1,e_2,e_3^\times)\cup
\CC(\H_1;\T_2,\T_1,e_1,e_2^\times, e_3)
}\tr_d(\tilde{\G})
\ge \sum_{\G \in \CC(\H_2;\T_2,\T_1,e_1,e_2,e_3^\times)\cup
\CC(\H_2;\T_2,\T_1,e_1,e_2^\times, e_3)}\tr_d(\G),
\end{equation}
with strict inequality if $d/m \ge 4$.

Case 3. $\G$ contains both $e_2$ and $e_3$.
Then $\G$ contains the cycle $C_3^m$.
Now $\tilde{\G},\G$ are respectively corresponding to $\H(v_1)\odot\T_2(w)$ and $\H(v_2)\odot\T_2(w)$ in Lemma \ref{cycleper}, and
$$ \tr_d(\tilde{\G}) \ge \tr_d(\G).$$
So we have
\begin{equation}\label{3cylce3}
\sum_{\tilde{\G} \in \CC(\H_1;\T_2,\T_1,e_1,e_2,e_3)
}\tr_d(\tilde{\G})
\ge \sum_{G \in \CC(\H_2;\T_2,\T_1,e_1,e_2,e_3)}\tr_d(G),
\end{equation}
with strict inequality if $d/m \ge 5$.

Combing Eqs. (\ref{3cylce1}), (\ref{3cylce2}) and (\ref{3cylce3}), we arrive at the inequality (\ref{ineg4}).
As $\CC(\H_1;\T_2,\T_1,e_1^\times)\ne \emptyset$ if $d/m \ge 2$, we get the desired result by Eq. (\ref{tradif4}).
\end{proof}

Denote $\mathcal{S}_{n,t}:=C_n(v) \odot S_t(w)$, where $v$ is a vertex of $C_n$ and $w$ is the center of $S_t$.
Then the $m$-th power of $\mathcal{S}_{n,t}$ is $\mathcal{S}^m_{n,t}:=C_n^m(v) \odot S_t^m(w)$, namely, a hypergraph obtained from $C_n^m$ by attaching $S_t^m$ with its center identified with a vertex of $C_n^m$ with degree $2$.

\begin{thm}
Let $\mathcal{U}$ be an $m$-uniform linear unicyclic hypergraph with $z$ edges and girth $3$.
Then
$$EE(\U)\leq EE(\mathcal{S}^m_{3,z-3}),$$
with equality if and only if $\U=\mathcal{S}^m_{3,z-3}$.
\end{thm}

\begin{proof}
Let $\U_{0}$ be an $m$-uniform linear unicyclic hypergraph with $z$ edges and girth $3$ with maximum Estrada index among all hypergraphs in $\mathscr{U}_{z,3}^m$.
By Lemma \ref{UatStar},
  we can write $\U_{0}$ as
  $$\U_{0}=C_{3}^{m}(v_1,v_2,v_3)\odot (S_{z_{1}}^{m}(v_{1}),S_{z_{2}}^{m}(v_{2}),S_{z_{3}}^{m}(v_{3})),$$
where $v_{i}$ is the vertex of $C_3^m$ with degree $2$ and also the center of $S_{z_{i}}^{m}$ for $i=1, 2, 3$, and some of the hyperstars $S_{z_{i}}^{m}$ may be trivial.
If $\U_0$ is not isomorphic to $\mathcal{S}^m_{3,z-3}$, then there exists two nontrivial hyperstars among $S_{z_{i}}^{m}$ for $i=1,2,3$.
Without loss of generality, let $S_{z_{1}}^{m}(v_{1}),S_{z_{2}}^{m}(v_{2})$ be nontrivial.
Now let $\tilde{\U}_0$ be obtained from $\U_0$ by relocating $S_{z_{2}}^{m}$ from $v_2$ to $v_1$.
By Corollary \ref{cycleTra}, we have
$$\Tr_d(\U_0) \le \Tr_d(\tilde{\U}_0),$$
with strict inequality if $m \mid d$ and $d/m \ge 2$.
So
$$EE(\U_{0})<EE(\tilde{\U}_0),$$
which yields a contradiction to the definition of $\U_{0}$.
So $\U_{0}=\mathcal{S}^m_{3,z-3}$.
The result follows.
\end{proof}

Du and Zhou \cite{DZ} proved that
$S_{3,z-3}$ is the unique unicyclic graph with maximum Estrada index among all unicyclic graphs with $z$ edges.
We have the following problem:

\begin{prob}
Is $\mathcal{S}^m_{3,z-3}$ the unique unicyclic hypergraph with maximum Estrada index among all linear unicyclic hypergraphs with $z$ edges?
\end{prob}

%%%%%%%%%%%%%%%%%%%%%%%%%%%%%%%%%%%%%%%%%%%%%%%%%%%%%%%%%%

\end{document}